\documentclass{amsproc}

\usepackage{aspmproc}
\usepackage{epsfig}

\usepackage{amscd,amssymb,amsmath}
\usepackage[all]{xy}

\def\comment#1{{\sf{[#1]}}}

\def\Z{{\mathbb Z}}
\def\Q{{\mathbb Q}}
\def\N{{\mathbb N}}
\def\R{{\mathbb R}}
\def\C{{\mathbb C}}
\def\D{{\mathbb D}}
\def\P{{\mathbb P}}
\def\bA{{\mathbb A}}

\def\L{{\mathbb L}}

\def\A{{\mathcal A}}
\def\F{{\mathcal F}}
\def\H{{\mathcal H}}
\def\I{{\mathcal I}}	%% could use Ch as in Tianjin paper
\def\cL{{\mathcal L}}

\def\O{{\mathcal O}}
\def\U{{\mathcal U}}
\def\cG{{\mathcal G}}
\def\cP{{\mathcal P}}

\def\cR{{\mathcal R}}

\def\a{{\mathfrak a}}
\def\f{{\mathfrak f}}
\def\g{{\mathfrak g}}
\def\h{{\mathfrak h}}
\def\o{{\mathfrak o}}
\def\p{{\mathfrak p}}
\def\r{{\mathfrak r}}

\def\t{{\mathfrak t}}
\def\u{{\mathfrak u}}

\def\G{{\Gamma}}

\def\Ahat{{\widehat{\A}}}
\def\Ghat{{\hat{\G}}}
\def\cGhat{{\widehat{\cG}}}

\def\Htilde{\widetilde{H}}
\def\gammatilde{{\tilde{\gamma}}}
\def\uhat{{\widehat{\u}}}

\def\IA{{I\!A}}
\def\IO{{I\!O}}
\def\cIA{{\I\!\A}}
\def\cIO{{\I\!\O}}
\def\ia{{\mathfrak{ia}}}
\def\io{{\mathfrak{io}}}

\def\Dtilde{{\widetilde{D}}}
\def\Ptilde{{\widetilde{P}}}
\def\Vtilde{{\widetilde{V}}}

\def\Rhat{{\widehat{R}}}
\def\Uhat{{\widehat{\U}}}

\def\Fbar{{\overline{F}}}

\def\rhotilde{{\tilde{\rho}}}
\def\thetatilde{{\tilde{\theta}}}
\def\phitilde{{\tilde{\phi}}}

\def\rhohat{{\hat{\rho}}}

\def\bmu{\pmb{\mu}}

\def\sl{\mathfrak{sl}}
\def\sp{\mathfrak{sp}}
\def\gl{\mathfrak{gl}}
\def\Sp{{\mathrm{Sp}}}
\def\SL{{\mathrm{SL}}}
\def\GL{{\mathrm{GL}}}

\def\SLhat{\widehat{\SL}}

\def\un{\mathrm{un}}

\def\fin{\mathrm{fin}}

\def\dot{{\bullet}}
\def\bs{\backslash}
\def\comp{~\widehat{\!}{\;}}

\def\ll{\langle\langle}
\def\rr{\rangle\rangle}

\def\triv{{\mathbf 1}}

\def\Vec{{\sf Vec}}

\newcommand\im{\operatorname{im}}               % image
\newcommand\id{\operatorname{id}}
\newcommand\ad{\operatorname{ad}}
\newcommand\Ad{\operatorname{Ad}}
\newcommand\coker{\operatorname{coker}}
\newcommand\Span{\operatorname{span}}
\newcommand\Spec{\operatorname{Spec}}
\newcommand\Hom{\operatorname{Hom}}
\newcommand\Homcts{\operatorname{Hom_{\mathrm{cts}}}}
\newcommand\End{\operatorname{End}}
\newcommand\Aut{\operatorname{Aut}}
\newcommand\Der{\operatorname{Der}}
\newcommand\Out{\operatorname{Out}}
\newcommand\OutDer{\operatorname{OutDer}}
\newcommand\Gr{\operatorname{Gr}}
\newcommand\Diff{\operatorname{Diff}}
\newcommand\codim{\operatorname{codim}}
\newcommand\Ind{\operatorname{Ind}}
\newcommand\Res{\operatorname{Res}}
\newcommand\Fun{\operatorname{Fun}}

\newtheorem{theorem}{Theorem}[section]
\newtheorem{lemma}[theorem]{Lemma}
\newtheorem{proposition}[theorem]{Proposition}
\newtheorem{corollary}[theorem]{Corollary}

\theoremstyle{definition}
\newtheorem{definition}[theorem]{Definition}
\newtheorem{example}[theorem]{Example}

\theoremstyle{remark}
\newtheorem{remark}[theorem]{Remark}

\newtheorem{convention}[theorem]{Convention}

\begin{document}

\title[Relative Weight Filtrations]{Relative Weight Filtrations on
Completions of Mapping Class Groups}

\dedicatory{To Professor Shigeyuki Morita on his 60th birthday}

\author{Richard Hain}
\address{Department of Mathematics\\ Duke University\\
Durham, NC 27708-0320}
\email{hain@math.duke.edu}

\thanks{Supported in part by the National Science Foundation through grants
DMS-0405440 and DMS-0706955.}

\date{\today}

%\subjclass{Primary xxxxx; Secondary xxxxx}

%\keywords{}

\maketitle

\tableofcontents

\section{Introduction}

One of the primary themes of Morita's work over the last 20 years has been the
study of the structure of mapping class groups via their actions on nilpotent
quotients of surface groups. A secondary theme has been the relation of this
work to the study of invariants of 3-manifolds and homology spheres. The goal of
this paper is to introduce to topologists a new tool that may be useful in these
pursuits.

The tool is the relative weight filtration of the relative completion of a
mapping class group of a surface that is associated to a system of simple closed
curves on the surface. Establishing the existence of relative weight filtrations
on completions of mapping class groups is non-trivial and was established with
Makoto Matsumoto using Galois actions in \cite{hain-matsumoto:mcg}, and with
Gregory Pearlstein and Tomohide Terasoma using Hodge theory in
\cite{hain-pearlstein}. The existence of relative weight filtrations on
completions of mapping class groups follows from general results about
fundamental groups of smooth varieties and because all mapping class groups
occur as fundamental group of moduli spaces of curves. The general theories of
the Hodge and Galois theory of fundamental groups of algebraic varieties imply
that relative weight filtrations have strong exactness properties.

The paper begins with an exposition for topologists of the theory of relative
completion of discrete groups. It is illustrated by the examples of mapping
class groups and automorphism groups of free groups. The paper continues with an
exposition of the weight filtration associated to a nilpotent endomorphism of a
vector space, and its generalization, the relative weight filtration associated
to a nilpotent endomorphism of a {\em filtered} vector space, which is due to
Deligne \cite{deligne:weil2} and was further developed by Steenbrink and Zucker
\cite{steenbrink-zucker}, and Kashiwara \cite{kashiwara}. Since the generic
nilpotent endomorphism of a filtered vector space does not have a relative
weight filtration, the existence of a relative weight filtration of a nilpotent
endomorphism of a filtered vector space imposes non-trivial restrictions on the
endomorphism. (Cf.\  \cite{steenbrink-zucker}.)

Relative weight filtrations appear in the study of mapping class groups in the
following context. Suppose that $S$ is a compact oriented surface of genus $g$
which, for simplicity, we suppose to be $\ge 3$. Denote the relative completion
of its mapping class group $\G_S$ by $\cG_S$. This is a proalgebraic group
(defined over $\Q$) that is an extension
$$
1 \to \U_S \to \cG_S \to \Sp(H_1(S)) \to 1
$$
of the symplectic group that is associated to the first homology of $S$ and its
intersection form by a prounipotent group $\U_S$, which is essentially (but not
quite) the unipotent completion of the Torelli group $T_S$ of $S$. (Cf.\
\cite{hain:torelli}.) There is a natural Zariski dense homomorphism $\G_S \to
\cG_S$; the image of the Torelli group $T_S$ is Zariski dense in $\U_S$. The Lie
algebra of $\cG_S$ is an extension
$$
0 \to \u_S \to \g_S \to \sp(H_1(S)) \to 0
$$
of the symplectic Lie algebra associated to $H_1(S)$ by $\u_S$, which is a
pronilpotent Lie algebra. It has a natural weight filtration which is defined by
$$
W_0\g_S = \g_S,\ W_{-1}\g_S = \u_S,\ W_{-m}\g_S = L^m\u_S\ (m\ge 1),
$$
where $L^m\u_S$ denotes the $m$th term of the lower central series of $\u_S$. A
system $\gamma = \{c_0,\dots,c_m\}$ of disjoint simple closed curves on $S$
determines commuting Dehn twists $\tau_0,\dots,\tau_m$. Their product
$\tau_\gamma$ lies in a prounipotent subgroup of $\cG_S$ and has a unique
logarithm $N_\gamma := \log \tau_\gamma \in \g_S$ whose adjoint action
$$
\ad(N_\gamma) : \g_S \to \g_S
$$
preserves the weight filtration $W_\dot$. It therefore induces an inverse system
of nilpotent endomorphisms
$$
\ad(N_\gamma) : \g_S/W_{-n}\g_S \to \g_S/W_{-n}\g_S
$$
each of which preserves the induced weight filtration $W_\dot$. General results
in Hodge and Galois theory imply the existence of the relative weight filtration
$M_\dot^\gamma$ of $\g_S$ associated to the curve system $\gamma$. This
filtration is compatible with the bracket in the sense that the bracket induces
a map
$$
M_r^\gamma\g_S \otimes M_s^\gamma\g_S \to M_{r+s}^\gamma \g_S.
$$
This filtration and the weight filtration $W_\dot$ have very strong exactness
and naturality properties. For example, there is a natural (though not
canonical) isomorphism of pronilpotent Lie algebras
$$
\g_S \cong \prod_{m,k} \Gr^{M^\gamma}_k\Gr^W_m \g_S.
$$
There are similar results when $S$ has decorations, such as points and boundary
components. The Lie algebra $\p(S,x)$ of the unipotent completion of
$\pi_1(S,x)$ also has a relative weight filtration associated to each curve
system $\gamma$ of $(S,x)$. The natural action $\g_S \to \Der \p(S,x)$ preserves
the relative weight filtration $M_\dot^\gamma$ as well as the weight filtration
$W_\dot$. This map has strong exactness properties with respect to both of these
filtrations. In particular, it is compatible with their identifications with
their associated bigraded objects.

Since the bracket preserves the relative weight filtration, $M^\gamma_k\g_S$ is
a subalgebra of $\g_S$ whenever $k\le 0$. These correspond to subgroups
$M_k^\gamma\cG_S$ ($k\le 0$) of $\cG_S$. The subalgebras $M^\gamma_0\g_S$,
parametrized by curve systems $\gamma$ on $S$ are natural analogues of parabolic
subalgebras of semi-simple Lie algebras as they are parametrized by the boundary
components of the corresponding moduli space of curves and they equal their
normalizers in $\g_S$, as we establish in Proposition~\ref{prop:parabolic}.

A particularly interesting case occurs when the curve system $\gamma$ is
maximal. Maximal curve systems correspond to pants decompositions. Since each
oriented ``pair of  pants'' naturally bounds a ball, a pants decomposition of
$S$ determines a handlebody $U$ with boundary $S$ in which each $c_j \in \gamma$
bounds an imbedded disk in $U$. The handlebody $U$ is unique in the sense that
if $V$ is another such handlebody, then there is a diffeomorphism $U \to V$
which restricts to the identity on $S$. In Section~\ref{sec:handle_filt} we use
Morse Theory to show that the relative weight filtration corresponding to a
pants decomposition depends only on the handlebody $U$ that it determines. We
denote it by $M_\dot^U$. In Section~\ref{sec:handle_gps} we use the exactness
properties of the relative weight filtration and results of Griffiths, Luft and
Pitsch \cite{griffiths,luft,pitsch} to show that the relative weight filtrations
$M_\dot^U$ and $M_\dot^V$ of $\g_S$ associated to two handlebodies $U$ and $V$
are equal if and only if there is a diffeomorphism $U \to V$ whose restriction
to their boundaries is the identity $S\to S$.

Each handlebody $U$ with boundary $S$ determines a subgroup $\Lambda_U$ of
$\G_S$ which consists of those elements of $\G_S$ that extend to an isotopy
class of diffeomorphisms of $U$. In the pointed case, we combine results of
Griffiths, Luft and Pitsch \cite{griffiths,luft,pitsch} with the exactness
properties of $M^U_\dot$ to prove that
$$
\Lambda_{U,x} = \G_{S,x} \cap M^U_0 \cG_{S,x},\
\Lambda_{U,x}\cap M^U_{-2}\cG_{S,x} = \ker\{\Lambda_{U,x} \to \Aut\pi_1(U,x)\}.
$$
These results provide an upper bound on the size of $\Lambda_{U,x}$ in
$\G_{S,x}$. They also imply that there is an injection
$$
\Aut\pi_1(U,x) \to \Gr^{M_U}_0\cG_{S,x}.
$$
This induces a homomorphism $\A_g \to \Gr^{M_U}_0\cG_{S,x}$ from the relative
completion $\A_g$ of $\Aut\pi_1(U,x)$. In Section~\ref{sec:handle_gps} we show
that this homomorphism is not surjective, which implies the unexpected result
that $\Lambda_{U}$ is not Zariski dense in $M_0^U\cG_{S}$. This implies that the
relative weight filtration of $\cG_{S}$ is not obtained simply by taking the
Zariski closure in $\cG_{S}$ of a filtration of $\G_{S}$.

We give an application to the problem of determining which elements of $\G_S$
extend to a handlebody. Very similar results have been obtained independently by
Jamie Jorgensen \cite{jorgensen} by different methods. If $S$ bounds the 
handlebody $U$, then the elements of $\G_S$ that extend to some handlebody is
$$
C = \bigcup_{\phi \in \G_S} \phi \Lambda_U \phi^{-1}.
$$
In Section~\ref{sec:applications}, we use properties of $M_\dot^U$ to show that
the Zariski closure of the intersection of $C$ with the $m$th term of the lower
central series of $\U_S$ is a proper (i.e., $\subsetneq$) closed subvariety of
the $m$th term of the lower central series of $\U_S$ for all $m\ge 1$ when $g\ge
7$ and slightly restricted ranges when $3\le g < 7$.

A more substantial potential application should be to finite type invariants of
3-manifolds and homology 3-spheres. The set of all genus $g$ Heegaard
decompositions of 3-manifolds and homology 3-spheres are the double coset spaces
$$
\Lambda_U \bs \G_S /\Lambda_U \text{ and } T\Lambda_V \bs T_S/T\Lambda_U,
$$
where $S^3 = U\cup V$ is a Heegaard decomposition of the 3-sphere, $\partial
U= \partial V = S$, and $T\Lambda_U := T_S\cap \Lambda_U$. The sets of all
3-manifolds and homology 3-manifolds are obtained from these by a suitable
stabilization where the genus of $S$ goes to infinity.\footnote{The
stabilization is by taking the connected sum with the standard genus one
Heegaard decomposition of the 3-sphere, cf.\ \cite{craggs}. To construct the
stabilization maps, one needs to use the mapping class group $\G_{S,D}$
associated to a surface with one boundary component and the corresponding
handlebody subgroup $\Lambda_{S,D}$.} It is thus natural to consider the double
coset spaces
$$
\cL_U \bs \cG_S / \cL_U \text{ and } W_{-1}\cL_V \bs \U_S / W_{-1}\cL_U,
$$
where $\cL_U$ denotes the Zariski closure of $\Lambda_U$ in $\cG_S$, as well as
their stabilizations as $g(S) \to \infty$. There are natural mappings
$$
\Lambda_U \bs \G_S /\Lambda_U \to \cL_U \bs \cG_S / \cL_U
$$
and
$$
T\Lambda_V \bs T_S/T\Lambda_U \to W_{-1}\cL_V \bs \U_S / W_{-1}\cL_U.
$$
and their stabilizations as $g(S) \to \infty$. Functions on $\cL_U \bs \cG_S /
\cL_U$ should yield finite type invariants of Heegaard decompositions and
functions on the stabilization should yield finite type invariants of
3-manifolds --- and similarly for homology 3-spheres. However, there is one
major difficulty in carrying out this program. One needs to take the quotients
$$
\cL_U \bs \cG_S / \cL_U
$$
using geometric invariant theory (GIT). But since the group $\cL_U$ is not
reductive, the GIT problems are more difficult and less likely to be well
behaved. (Cf. \cite{git}.)  Amassa Fauntleroy and I are attempting to use the
strictness properties of $M^U_\dot$ to construct and study the GIT quotients
above.

The reader should be aware that, in an attempt to make this material more
accessible to non-experts, the Hodge and Galois theoretic aspects of the theory
have been suppressed. This choice comes at the expense of giving the basic
properties of relative weight filtrations a false aura of mystery. Readers
wanting more background should consult the papers of Deligne
\cite{deligne:weil2}, Steenbrink-Zucker \cite{steenbrink-zucker} and Kashiwara
\cite{kashiwara}.

\begin{convention}
\label{convention}
Throughout the paper the default coefficient group in all homology and
cohomology groups is $\Q$, the rational numbers. So, for example, $H_\dot(S)$
denotes the rational homology of $S$ and $H^\dot(\G)$ denotes the rational
cohomology of the group $\G$. All other coefficients will be made explicit.
\end{convention}

\bigskip

\noindent{\em Acknowledgments:} First and foremost, it is a great pleasure to
thank Shigeyuki Morita for his decades of inspiring work. His work complements
my own and has had significant impact upon it. I would also like to thank my
collaborators Gregory Pearlstein, Tomohide Terasoma and especially Makoto
Matsumoto for joint work which makes the current work possible. I would also
like to thank Alan Hatcher for communicating a proof of the folk result,
Proposition~\ref{prop:folk} and Amassa Fauntleroy for discussions of geometric
quotients by non-reductive groups. Finally, I thank the anonymous referee for
his very careful and reading of the manuscript and for his suggested
improvements to Proposition~\ref{prop:fte_ext}.

\section{Filtrations}

Since filtrations play a central role in this paper, it is wise to first lay out
the general conventions used in this paper. Deligne's conventions on filtrations
\cite{deligne:hodge2} are used systematically as they work well and as they are
used in Hodge theory and the study of Galois actions.

Suppose that $V$ is a vector space over a field $F$ of characteristic zero. An
increasing filtration $G_\dot$ of $V$ is a sequence of subspaces
$$
\cdots \subseteq F_{m-1} V \subseteq F_m V \subseteq F_{m+1} V \subseteq \cdots
$$
where $m\in \Z$. When $V$ is finite dimensional, we require that the
intersection of the $F_m V$ be trivial and that their union be all of $V$. The
infinite dimensional case is more subtle and is discussed  below.

The $m$th graded quotient $F_mV/F_{m-1}V$ of $F_\dot$ will be denoted by
$\Gr^F_m V$. The associated graded vector space will be denoted by $\Gr^F_\dot
V$.

Decreasing filtrations of $V$ will be denoted with an upper index:
$$
\cdots \supseteq F^{m-1} V \supseteq F^m V \supseteq F^{m+1} V \supseteq \cdots
$$
The $m$th graded quotient $F^mV/F^{m+1}V$ will be denoted by $\Gr_F^\dot V$. We
require that the intersection of the $F^m V$ be trivial and that their union be
all of $V$.

An increasing filtration $F_\dot$ of $V$ can be regarded as a decreasing
filtration by ``raising indices'': $F^mV := F_{-m}V$. For this reason, we will
discuss the remaining properties only for increasing filtrations.

If $(V,F_\dot)$ and $(W,F_\dot)$ are filtered vector spaces, then $V\otimes W$
and $\Hom(V,W)$ inherit natural filtrations:
$$
F_m (V\otimes W) := \sum_{j+k=m} F_jV\otimes F_k W
$$
and
$$
F_m\Hom(V,W) :=
\{\phi : V \to W : \phi(F_k V)\subseteq F_{m+k}W\text{ for all }k\in \Z\}.
$$
In particular, the dual $V^\ast$ of $V$ has a natural filtration
$$
F_m V^\ast = \{\phi \in V^\ast : \phi(F_{-m-1}V) = 0\}.
$$
With these definitions, there are natural isomorphisms
\begin{align*}
\Gr^F_m (V\otimes W) &\cong \bigoplus_{j+k=m}\Gr^F_jV\otimes \Gr^F_kV \cr
\Gr^F_m\Hom(V,W) &\cong \bigoplus_k \Hom(\Gr^F_k V,\Gr^F_{m+k}W)\cr
\Gr^F_m \big(V^\ast\big) &\cong (\Gr^F_{-m}V)^\ast.
\end{align*}

A filtration $F_\dot$ of $V$ naturally induces one on every subspace $W$ and
every quotient $p : V \to V/U$ by
$$
F_m W := W\cap F_mV \text{ and } F_m (V/U) := p(F_mV).
$$
There are thus two ways of inducing a filtration on a subquotient $q : W \to
p(W)$. One way is to restrict the quotient filtration on $V/U$ to the subspace
$p(W)$; the other is to give $p(W)$ the image of the filtration induced by
$F_\dot$ on $W$. These are easily seen to agree. (Cf.\ \cite{deligne:hodge2}).

In particular, if a vector space $V$ has two filtrations $F_\dot$ and $G_\dot$,
then the filtration $F_\dot$ induces a natural filtration (also denoted by
$F_\dot$) on each $G_\dot$-graded quotient $\Gr^G_m V$ of $V$. There are then
natural isomorphisms
$$
\Gr^F_m \Gr^G_n V \cong \Gr^G_n \Gr^F_m V.
$$

\subsection{The infinite dimensional case}
\label{sec:inf_dim}

Unless otherwise noted, all infinite dimensional vector spaces considered will
be either ind- or pro- objects of the category $\Vec_F^\fin$ of finite
dimensional $F$-vector spaces. The dual of a pro-object of $\Vec_F^\fin$ is an
ind-object of $\Vec_F^\fin$, and vice-versa. If $V$ is an ind-object of
$\Vec_F^\fin$, then it is naturally isomorphic to its double dual. Similarly, an
ind-object of $\Vec_F^\fin$ is naturally isomorphic to its double dual. A
filtration of a pro-object (resp.\ ind-object) of $\Vec_F^\fin$ is simply a
filtration of it by pro-objects (resp.\ ind-objects) of $\Vec_F^\fin$.

\section{Relative Completion of Discrete Groups}

Here we summarize the theory of relative unipotent completion of discrete
groups. Some of the statements are stronger than results in the literature. Full
proofs will appear in \cite{hain-matsumoto:notes}. Versions of many of these
results for the related notion of weighted completion can be found in
\cite{hain-matsumoto:deligne}.

\subsection{Unipotent and Prounipotent Groups}

Suppose that $F$ is a field of characteristic zero. Recall that a unipotent
algebraic group over $F$ is a subgroup $U$, for some $n$, of the group of the
$n\times n$ unipotent upper triangular matrices
$$
\{X \in \GL_n(F) : X-I \text{ is strictly upper triangular}\}
$$
that is defined by polynomial equations. The Baker-Campbell-Hausdorff formula
implies that
$$
\u = \{\log u \in \gl_n(F): u \in U\}
$$
is a Lie algebra with bracket $[x,y] = xy-yx$. The exponential mapping $\exp :
\u \to U$ is a polynomial bijection. The algebraic subgroups of $U$ correspond
bijectively to the Lie subalgebras of $\u$ via the exponential mapping.

A {\em pronilpotent} Lie algebra is, by definition, the inverse limit of
finite dimensional nilpotent Lie algebras:
$$
\u = \varprojlim_{\alpha} \u_\alpha
$$
It has a natural topology; a base of neighbourhoods of 0 consists of the
kernels of the projections of $\u$ to each of the $\u_\alpha$.

A {\em prounipotent group} $\U$ is the inverse limit
$$
\U = \varprojlim_{\alpha} U_\alpha
$$
of an inverse system of unipotent groups. The Lie algebra $\u$ of $\U$ is the
inverse limit of the Lie algebras of the $U_\alpha$. It is a pronilpotent Lie
algebra  The exponential mapping $\exp : \U \to \u$ is an isomorphism of
proalgebraic varieties.

Pronilpotent Lie algebras have nice presentations. Suppose that $\u$ is a
pronilpotent Lie algebra. Define $H_1(\u)$ to be the abelianization of $\u$. It
is a topological vector space --- if $\u$ is the inverse limit of the finite
dimensional nilpotent Lie algebras $\u_\alpha$, then $H_1(\u)$ is the inverse
limit of the $H_1(\u_\alpha)$. A continuous section $s : H_1(\u) \to \u$ of the
natural projection $\u \to H_1(\u)$ induces a continuous Lie algebra
homomorphism
$$
s_\ast : \L(H_1(\u))\comp \to \u
$$
from the free completed Lie algebra generated by $H_1(\u)$ to $\u$. This induces
an isomorphism on abelianizations. Since $\u$ is pronilpotent, $s_\ast$ is
surjective. It follows that $\u$ has a presentation of the form
$$
\u \cong \L(H_1(\u))\comp/\r
$$
where $\r = \ker s_\ast$ is a closed ideal contained in the commutator
subalgebra of $\L(H_1(\u))\comp$. Such presentations are said to be {\em
minimal}.

Define the continuous cohomology $H^\dot(\u)$ of a pronilpotent Lie algebra $\u$
that is the inverse limit of the finite dimensional nilpotent Lie algebras
$\u_\alpha$
by
$$
H^\dot(\u) := \varinjlim_{\alpha} H^\dot(\u_\alpha).
$$
This will be regarded an ind-object of the category of finite dimensional
$F$-vector spaces. It is easy to check that for all $k\ge 0$,
$$
H^k(\u) = \Homcts(H_k(\u),F) \text{ and }H_k(\u) = \Hom(H^k(\u),F).
$$
This generalizes to pronilpotent coefficients (i.e., projective systems of
nilpotent coefficients): If $V = \varprojlim V_\alpha$ where $V_\alpha$ is a
nilpotent $\u_\alpha$-module,
then
$$
H_k(\u,V) := \varprojlim_\alpha H_k(\u_\alpha,V_\alpha)
\text{ and }
H^k(\u,V^\ast) := \varinjlim_\alpha H^k(\u_\alpha,V_\alpha^\ast).
$$
There are natural isomorphisms
$$
H^k(\u,V^\ast) = \Homcts(H_k(\u,V),F) \text{ and }
H_k(\u,V) = \Hom(H^k(\u,V^\ast),F)
$$
If $\f$ is a free pronilpotent Lie algebra, then $H^k(\f;V) = 0$ for all $k>1$
and all nilpotent $\f$-modules $V$.

A complete proof of the following analogue of Hopf's Theorem will appear in
\cite{hain-matsumoto:notes}.

\begin{proposition}
If $\r$ is a closed ideal in the free pronilpotent Lie algebra $\f$ that is
contained in $[\f,\f]$, then there is a natural isomorphism
$$
H^2(\f/\r) \cong \Homcts(\r/[\r,\f],F)
$$
of ind-vector spaces. Moreover, if $\theta:\r/[\r,\f] \to \r$ is a continuous
section of the quotient mapping, then $\r$ is generated as a closed ideal by
$\im \theta$.
\end{proposition}

\begin{proof}
The fact that every subalgebra of a free Lie algebra is free \cite{reutenauer},
implies that every subalgebra of a free pronilpotent Lie algebra is also a free
pronilpotent Lie algebra. Consequently, $\r$ is a free pronilpotent Lie algebra
and has vanishing cohomology in degrees $>1$. Using standard cochains, one can
show that there is a spectral sequence
$$
E_2^{s,t} = H^s(\f/\r,H^t(\r)) \Rightarrow H^{s+t}(\f).
$$
Since $\r \subseteq [\f,\f]$, $H^1(\f) = H^1(\f/\r)$. Since $H^1(\r) =
\Homcts(H_1(\r),F)$ and since
$$
H_0(\f/\r,H_1(\r)) = H_0(\f,H_1(\r)) \cong \r/[\r,\f],
$$
the vanishing of the higher cohomology of $\f$ and $\r$ imply (when plugged into
the spectral sequence) that
$$
H^2(\f/\r) = H^0(\f,H^1(\r)) = \Homcts(H_0(\f,H_1(\r)),F) = \Homcts(\f/[\r,\f],F).
$$
\end{proof}

An immediate corollary is an analogue of Stallings' result \cite{stallings}. A
detailed proof will appear in \cite{hain-matsumoto:notes}.

\begin{corollary}
\label{cor:isom}
A homomorphism $\phi : \u_1 \to \u_2$ of pronilpotent Lie algebras is an
isomorphism if and only if it induces an isomorphism $H^1(\u_2) \to H^1(\u_1)$
and a monomorphism $H^2(\u_2) \to H^2(\u_1)$ of ind-vector spaces.
\end{corollary}

\begin{proof}[Sketch of Proof]
The only if assertion is trivially true. Suppose that $H^k(\u_2)\to H^k(\u_1)$
is an isomorphism when $k=1$ and a monomorphism when $k=2$. Since $\u_1$ and
$\u_2$ are pronilpotent, the isomorphism on $H^1$ implies that $\phi$ is a
quotient map in the category of pronilpotent Lie algebras. Choose a minimal
presentation $\u_1 = \f/\r_1$. Let $\r_2$ be the kernel of $\f \to \u_1 \to
\u_2$. Then $\u_2 = \f/\r_2$ and $\phi$ is an isomorphism if and only if the
inclusion $\phi : \r_1 \to \r_2$ is a quotient mapping. But this holds if and
only if
$$
H^2(\u_2) = \Homcts(\r_2/[\r_2,\f],F) \to \Homcts(\r_1/[\r_1,\f],F) = H^2(\u_1)
$$ 
is a monomorphism.
\end{proof}

\begin{corollary}
\label{cor:triv-free}
A pronilpotent Lie algebra $\u$ is trivial if and only if $H^1(\u) = 0$ and
free if and only if $H^2(\u) = 0$. $\Box$
\end{corollary}

\subsection{Relative Unipotent Completion}

The data for relative completion are:
\begin{enumerate}

\item a discrete group $\G$;

\item a field $F$ of characteristic zero;

\item a reductive algebraic group $R$ over $F$, such as $\GL_n(F)$, $\SL_n(F)$
or
$\Sp_n(F)$;

\item a Zariski dense homomorphism $\rho : \G \to R$.\footnote{In this paper I
will not distinguish between an algebraic group $G$ over $F$ and its group of
$F$- rational points $G(F)$.}

\end{enumerate}

The completion of $\G$ with respect to $\rho$ consists of a proalgebraic group
(i.e., an inverse limit of algebraic groups) $\cG$ over $F$ that is an extension
\begin{equation}
\label{eqn:extn}
1 \to \U \to \cG \to R \to 1
\end{equation}
where $\U$ is prounipotent and a homomorphism $\rhohat : \G \to \cG$ whose
composition with $\cG \to R$ is $\rho$. It is characterized by the following
universal mapping property:

If $G$ is an affine (pro)algebraic group over $F$ that is an extension
$$
1 \to U \to G \to R \to 1
$$
of $R$ by a (pro)unipotent group $U$, and if $\rhotilde : \G \to G$ is a
homomorphism whose composition with $G \to R$ is $\rho$, then there is a unique
homomorphism $\phi : \cG \to G$ of (pro)algebraic $F$-groups such that
$$
\xymatrix{
\G \ar[r]^\rhohat \ar[d]_\rhotilde & \cG \ar[d]\ar[dl]_\phi \cr
G \ar[r] & R
}
$$
commutes.

The universal mapping property implies that the homomorphism $\rhohat : \G \to
\cG$ is Zariski dense --- that is, if $\cG'$ is a proalgebraic subgroup of $\cG$
defined over $F$ that contains the image of $\rhohat$, then $\cG'=\cG$. The
point being that the Zariski closure of $\im\rhohat$ in $\cG$ has the same
universal mapping property as $\cG$.

Suppose that $K$ is an extension field of $F$. Every (pro)algebraic group $G$
over $F$ gives rise to a (pro)algebraic group $G\otimes_F K$ over $K$ by
extension of scalars. The universal mapping property of the relative completion
$\cG_K$ of $\G$ over $K$ with respect to $\rho : \G \to R\otimes_F K$ implies
that $\G \to \cG\otimes_F K$ induces a homomorphism $\cG_K \to \cG\otimes_F K$.

\begin{theorem}[Hain-Matsumoto \cite{hain-matsumoto:notes}]
\label{thm:base-change}
The homomorphism $\cG_K \to \cG\otimes_F K$ is an isomorphism.
\end{theorem}

In general we will work over the smallest field $F$ possible, which is the
smallest field over which both $R$ and $\rho$ are defined. In all principal
examples in this paper, the field will be $\Q$.

\subsection{Unipotent Completion}

When $R$ is the trivial group, relative completion reduces to classical
unipotent completion, which is also known as Malcev completion and which can be
computed by the methods of rational homotopy theory due to Quillen, Chen and
Sullivan. We shall denote the unipotent completion of $\G$ over $F$ by
$\G^\un_{/F}$. The default field will be $\Q$. We shall abbreviate
$\G^\un_{/\Q}$ to $\G^\un$. The unipotent completion of $\G$ over $F$ is
obtained from $\G^\un$ by extension of scalars:
$$
\G_{/F}^\un = \G^\un \otimes_\Q F.
$$

If $\G$ is a free group $F_n = \langle x_1,\dots,x_n \rangle$ on $n$-generators,
the Lie algebra of $F^\un_{n/F}$ is the completed free Lie algebra
$$
\f_n := \L(X_1,\dots,X_n)\comp
$$
which is the closure of the free Lie algebra $\L(X_1,\dots,X_n)$ in the
noncommutative power series ring $F\ll X_1,\dots,X_n\rr$. The prounipotent
group $F_{n/F}^\un$ is
$$
F_{n/F}^\un = \{\exp u \in F\ll X_1,\dots,X_n\rr : u \in \f_n\}.
$$
The natural homomorphism $\rhohat : F_n \to F_{n/F}^\un$ is defined by
$\rhohat(x_j) = \exp X_j$. This can be proved using universal mapping
properties. A theorem of Magnus \cite{magnus} implies that the homomorphism $F_n
\to F_n^\un$ is injective.

\subsection{Completions of $\Aut F_n$ and $\Out F_n$}
\label{sec:aut_out}

Denote the automorphism group of the free group $F_n$ by $\Aut F_n$ and its
quotient by inner automorphisms of $F_n$ by $\Out F_n$. There are natural
surjections
$$
\Aut F_n \to \GL_n(\Z) \text{ and } \Out F_n \to \GL_n(\Z).
$$
Denote their kernels by $\IA_n$ and $\IO_n$, respectively. Let $\Aut^+ F_n$ be
the index 2 subgroup of $\Aut F_n$ whose image in $\GL_n(\Z)$ is $\SL_n(\Z)$.
Let $\Out^+F_n$ be its quotient by the group of inner automorphisms.

Take $F=\Q$, $R = \SL_n(\Q)$ and let
$$
\rho : \Aut^+ F_n \to \SL_n(\Q)
$$
be the natural representation. This is Zariski dense. The completion of $\Aut^+
F_n$ with respect to $\rho$ is an extension
$$
1 \to \cIA_n \to \A_n \to \SL_n(\Q) \to 1.
$$
Similarly, we have the completion of $\Out^+F_n$. It is an extension
$$
1 \to \cIO_n \to \O_n \to \SL_n(\Q) \to 1.
$$
The corresponding sequences of Lie algebras
$$
0 \to \ia_n \to \a_n \to \sl_n \to 0
\text { and }
0 \to \io_n \to \o_n \to \sl_n \to 0
$$
are exact. There are natural homomorphisms
$$
\a_n \to \Der \f_n \text{ and } \o_n \to \OutDer \f_n,
$$
where $\OutDer \f_n$ denotes the Lie algebra of outer derivations of $\f_n$.

The natural homomorphisms $\IA_n \to \cIA_n$ and $\IO_n \to \cIO_n$ induce
homomorphisms $\IA_n^\un \to \cIA_n$ and $\IO_n^\un \to \cIO_n$. We shall see
later (cf.\ Cor.~\ref{cor:IA_IO_unipt}) that these homomorphisms are
isomorphisms when $n\ge 4$, surjective when $n=3$ and are far from surjective
when $n=2$.

\begin{proposition}
The natural homomorphism $\rhohat : \Aut^+ F_n \to \A_n$ is injective.
\end{proposition}

\begin{proof}
Since the unipotent completion $F_n \to F_n^\un$ is injective, it follows that
the natural representation $\theta : \Aut F_n \to \Aut F_n^\un$ is injective.
The Zariski closure of the image of $\Aut^+ F_n$ under $\theta$ is easily seen
to be an extension of $\SL_n(\Q)$ by a prounipotent group. The universal mapping
property of relative completion induces a homomorphism $\psi : \A_n \to \Aut
F_n^\un$ such that the diagram
$$
\xymatrix{
\Aut^+ F_n \ar[r]^(0.6)\rhohat\ar[dr]_\theta & \A_n \ar[d]^\psi \cr
& \Aut F_n^\un
}
$$
The injectivity of $\rhohat$ follows from the injectivity of $\theta$.
\end{proof}

The injectivity of $\Out^+ F_n \to \Out F_n^\un$ would follow if one could prove
that $F_n^\un \cap \Aut F_n = F_n$ in  $\Aut F_n^\un$, where $F_n$ and $F_n^\un$
are regarded as subgroups of $\Aut F_n^\un$ via the inner action. This is not
clear.

\subsection{Properties of Relative Completion}

Here we list some of the basic properties of relative completion that we shall
need. The proofs of these are sprinkled throughout the literature and are
sometimes proved for the related notion of {\em weighted completion}
\cite{hain-matsumoto:deligne}. The notes \cite{hain-matsumoto:notes} will give
an efficient and uniform presentation of the theory relative and related
completions of discrete and profinite groups.

\begin{proposition}[Naturality]
\label{prop:naturality}
Suppose that $\rho_j : \G_j \to R_j$, $j=1,2$ are Zariski dense homomorphisms
from discrete groups to reductive groups over $F$. Let $\G_j \to \cG_j$ be the
completion of $\G_j$ with respect to $\rho_j$ over $F$. If the diagram
$$
\xymatrix{
\G_1 \ar[r]^{\rho_1}\ar[d]_{\phi_\G} & R_1 \ar[d]^{\phi_R} \cr
\G_2 \ar[r]^{\rho_2} & R_2
}
$$
commutes where $\phi_R$ is a homomorphism of algebraic groups, then there is
a unique homomorphism $\phi_\cG : \cG_1 \to \cG_2$ such that the diagram
$$
\xymatrix{
\G_1 \ar[r]^{\rhohat_1}\ar[d]_{\phi_\G} & \cG_1 \ar[d]^{\phi_\cG} \ar[r]
& R_1 \ar[d]^{\phi_R} \cr
\G_2 \ar[r]^{\rhohat_2} & \cG_2 \ar[r] & R_2
}
$$
\end{proposition}

Completions are, in general, right exact. Here we state a useful special case.

\begin{proposition}[Right exactness]
\label{prop:rightexact}
Suppose that $\rho : \G \to R$ is a Zariski dense homomorphism from a discrete
group to a reductive $F$-group. Denote the completion $\G$ with respect to
$\rho$ by $\cG$ and the completion of $\im \rho$ with respect to the inclusion
$\im\rho  \hookrightarrow R$ by $\cR$. Then the sequence
$$
\big(\ker\rho\big)^\un_{/F} \to \cG \to \cR \to 1
$$
is exact.
\end{proposition}

A generalization of Levi's Theorem  implies that, when $\G$ is finitely
generated, the extension (\ref{eqn:extn}) is split, and that any two splitting
are conjugate by an element of $\U$. It follows that $\u$ is a Lie algebra in
the category of pro-representations\footnote{That is, an inverse limit of finite
dimensional $R$-modules. Since $R$ is reductive, the pro-representations of $R$
are direct products of finite dimensional $R$-modules.} of $R$ and that there is
an isomorphism
$$
\cG \cong R \ltimes \exp \u
$$
that is unique up to conjugation by an element of $\exp\u$.

Relative completions are manageable and somewhat computable as they are quite
tightly controlled by cohomology.

Suppose that $\Fbar$ is an algebraic closure of $F$. An irreducible
representation $V$ of $R$ is {\em absolutely irreducible} if $V\otimes_F\Fbar$
is an irreducible representation of $R\otimes_F \Fbar$.

\begin{theorem}
\label{thm:presentation}
For all finite dimensional $R$-modules $V$, there is a homomorphism
$$
\Hom_R(H_k(\u),V) \cong \big(H^k(\u)\otimes V\big)^R \to H^k(\G;V)
$$
that is natural with respect to the maps described in
Proposition~\ref{prop:naturality} It is an isomorphism when $k=1$ and injective
when $k=2$. If every irreducible finite dimensional representation of $R$ is
absolutely irreducible, then there is a natural $R$-module isomorphism
$$
H^1(\u) \cong \bigoplus_\alpha H^1(\G;V_\alpha)\otimes V_\alpha^\ast
$$
and a natural $R$-module injection
$$
H^2(\u) \hookrightarrow \bigoplus_\alpha H^2(\G;V_\alpha)\otimes V_\alpha^\ast
$$
where $\{V_\alpha\}$ is a set of representatives of the isomorphism classes of
irreducible finite dimensional $R$-modules and where each $H^1(\G;V_\alpha)$
is regarded as a trivial $R$-module.
\end{theorem}

This theorem alone and in combination with the Base Change
Theorem~\ref{thm:base-change} can often be used to compute $\u$. Combined with
Corollary~\ref{cor:triv-free}, it gives the following criterion for the
vanishing of $\u$.

\begin{corollary}
\label{cor:vanishing}
The prounipotent radical of $\cG$ vanishes if and only if $H^1(\G;V) = 0$ for
all irreducible finite dimensional $R$-modules. $\Box$
\end{corollary}

Combining this with right exactness (Prop.~\ref{prop:rightexact}) yields:

\begin{corollary}
\label{cor:surjectivity}
If $H^1(\im\rho;V) = 0$ for all finite dimensional $R$-modules $V$, then
$$
\big(\ker\rho\big)^\un_{/F} \to \cG \to R \to 1
$$
is exact. That is, $\U$ is a quotient of $\big(\ker\rho\big)^\un_{/F}$.
\end{corollary}

Additional hypotheses give an upper bound on the kernel. The following result is
a refined version of \cite[Prop.~4.13]{hain:comp}.

\begin{theorem}
\label{thm:kernel}
Suppose that $\rho : \G \to R$ is a Zariski dense homomorphism. Denote $\ker
\rho$ by $T$. If the $\im \rho$ module $H_1(T;F)$ is finite dimensional and the
restriction (via $\rho$) of a finite dimensional $R$-module, and if $H^1(\im
\rho;V) = 0$ for all irreducible finite dimensional representations of $R$, then
there is a natural exact sequence
$$
1 \to K \to T^\un_{/F} \to \cG \to R \to 1
$$
where $K$ is contained in the center of $T^\un_{/F}$. Moreover, if
$H^2(\im\rho;V)$ is finite dimensional for all irreducible $R$-modules $V$,
then $K$ is an $R$-submodule of the abelian prounipotent group
$$
\prod_{\alpha} H^2(\im \rho;V_\alpha)^\ast \otimes V_\alpha.
$$
where $V_\alpha$ ranges over representatives of the isomorphism classes of
finite dimensional $R$-modules.
\end{theorem}

\subsection{Examples}
Equipped with the results of the previous section, we can approach the problem
of computing the relative completions in natural examples. 

\begin{example}[Lattices]
\label{ex:lattices}
If $\G$ is an irreducible lattice in a semi-simple real Lie group $G$ of rank
$\ge 2$, then Raghunathan's vanishing theorem \cite{raghunathan} states that
$$
H^1(\G;V) = 0
$$
for all irreducible representations $V$ of $G$. Corollary~\ref{cor:vanishing}
implies that the completion of $\G$ with respect to the inclusion $\G \to G$
over $\R$ is $G$.

In particular, when $n\ge 3$, the completion of any finite index subgroup $\G$
of $\SL_n(\Z)$ with respect to the inclusion $\G \to \SL_n(\Q)$ is $\SL_n(\Q)$.
When $g\ge 2$, the completion of any finite index subgroup of $\Sp_g(\Z)$ with
respect to the inclusion $\G \to \Sp_g(\Q)$ is $\Sp_g(\Q)$.

The rank condition is necessary. The groups $\SL_2(\R)$ and $\Sp_1(\R)$ are
isomorphic and have real rank 1. If we take $\G$ to be one of the isomorphic
groups $\SL_2(\Z)$, $\Aut^+ F_2$, $\G_{S,x}$, where $S$ is a genus 1 surface,
then the prounipotent radical of the completion of $\G$ with respect to the
inclusion $\G \to \SL_2(\Q)$ is a free prounipotent group whose abelianization
is infinite dimensional. (Cf. \cite[Rem.~3.9]{hain:torelli}.) It is closely
connected with classical modular forms and elliptic motives. (Cf.\
\cite{hain-pearlstein}.)
\end{example}

Results of Borel \cite{borel} imply that if $\im \rho$ is arithmetic of
sufficiently high rank, then $H^2(\im\rho;V)$ vanishes for all non-trivial
$R$-modules and $H^2(\im \rho;\Q)$ is isomorphic to the corresponding cohomology
group of the compact dual of the symmetric space of $R\otimes\R$. In particular,
Borel's formula implies the vanishing of $H^2(\SL_2(\Z);V)$ for {\em all}
$\SL_n$-modules $V$ when $n\ge 4$. It also implies the vanishing of
$H^2(\Sp_g(\Z),V)$ for all non-trivial irreducible $\Sp_g$-modules when $g\ge
3$.

\begin{example}[Universal Central Extensions]
Suppose that $\G$ is a non-zero multiple of the universal central extension of
$\Sp_g(\Z)$, where $g \ge 2$:
$$
0 \to \Z \to \G \to \Sp_g(\Z) \to 1.
$$
Let $R = \Sp_g(\Q)$ and $\rho : \G \to \Sp_g(\Q)$ be the obvious homomorphism.
Denote the relative completion of $\G$ with respect to $\rho$ by $\cG$. By
Example~\ref{ex:lattices}, the completion of $\Sp_g(\Z)$ with respect to the
inclusion $\Sp_g(\Z) \to \Sp_g(\Q)$ is $\Sp_g(\Q)$. Raghunathan's Theorem
implies that $H^1(\Sp_g(\Z);V)$ vanishes for all finite dimensional
representations $V$ of $\Sp_g$. An elementary spectral sequence argument implies
that $H^1(\G,V)$ also vanishes for all finite dimensional $\Sp_g$-modules $V$.
Cor.~\ref{cor:vanishing} then implies that $\cG \to \Sp_g(\Q)$ is an
isomorphism.

This provides an interesting example of Theorem~\ref{thm:kernel}. Borel's
vanishing theorem implies that, when $g\ge 3$, $H^2(\Sp_g(\Z),V)$ vanishes
for all non-trivial irreducible $\Sp_g(\Q)$-modules and that $H^2(\Sp_g(\Z),\Q)$
is 1-dimensional.
Since the
unipotent completion of $\Z$ is $\Q$, Theorem~\ref{thm:kernel} implies that we
have an exact sequence
$$
H^2(\Sp_g(\Q);\Q)^\ast \to \Q \to \cG \to \Sp_g(\Q) \to 1. 
$$
Since $\cG \to \Sp_g(\Q)$ is an isomorphism, it follows that
$H^2(\Sp_g(\Q);\Q)^\ast \to \Q$ is an isomorphism and that $\Q \to \cG$ is
trivial. $\Box$
\end{example}

As remarked in Example~\ref{ex:lattices}, $\IA_2^\un \to \cIA_2$ and $\IO_2^\un
\to \cIO_2$ are far from surjective. However, when $n\ge 3$, the situation
improves.

\begin{corollary}
\label{cor:IA_IO_unipt}
If $n\ge 3$, then the natural homomorphisms $\IA_n^\un \to \cIA_n$ and
$\IO_n^\un \to \cIO_n$ are surjective. If $n\ge 4$, they are isomorphisms. 
\end{corollary}

\begin{proof}
By results of Magnus \cite{magnus} and Kawazumi \cite{kawazumi}, there are
natural $\GL_n(\Z)$-equivariant isomorphisms
$$
H_1(\IA_n) \cong \Hom(V,\Lambda^2 V) \text{ and }
H_1(\IO_n) \cong \Hom(V,\Lambda^2 V)/V,
$$
where $V = H_1(F_n)$, from which it follows that the $\SL_n(\Z)$-modules
$H_1(\IA_n)$ and $H_1(\IO_n)$ are the restrictions of $\SL(V)$-modules.
Surjectivity when $n\ge 3$ follows from Corollary~\ref{cor:surjectivity} and
Raghunathan's vanishing result. When $n\ge 4$, the result follows from
Theorem~\ref{thm:kernel} and Borel's vanishing result, stated above.
\end{proof}

Another situation in which left exactness holds, that we shall need later, is
the following. Suppose that
$$
\xymatrix{
1 \ar[r] & \G \ar[d]_{\rho_\G} \ar[r] & \Ghat \ar[d]_{\rho_\Ghat}\ar[r] &
G \ar@{=}[d] \ar[r] & 1 \cr
1 \ar[r] & R \ar[r] & \Rhat \ar[r] & G \ar[r] & 1
}
$$
is a commutative diagram of groups with exact rows where:
\begin{enumerate}

\item $\G$ and $\Ghat$ are discrete groups;

\item $R$ and $\Rhat$ are reductive $F$-groups;

\item $G$ is a finite group;

\item $\rho_\G$ is Zariski dense (which implies that $\rho_{\Ghat}$ is also
Zariski dense).

\end{enumerate} 
Denote the completion of $\G$ with respect to $\rho_\G$ by $\cG$ and the
completion of $\Ghat$ with respect to $\rho_\Ghat$ by $\cGhat$. Naturality
implies that there is a homomorphism $\cG \to \cGhat$ such that the diagram
$$
\xymatrix{
\G \ar[r] \ar[d] & \cG \ar[d] \cr
\Ghat \ar[r] & \cGhat
}
$$
commutes. Right exactness implies that the sequence
$
\cG \to \cGhat \to G \to 1
$
is exact. Denote the prounipotent radicals of $\cG$ and $\cGhat$ by $\U$
and $\Uhat$, respectively.

\begin{proposition}
\label{prop:fte_ext}
The natural homomorphism $\cG \to \cGhat$ is injective. Consequently, the
induced mapping $\U \to \Uhat$ of prounipotent radicals is an isomorphism.
\end{proposition}

\begin{proof}
Stallings' criterion (Cor.~\ref{cor:isom}) will be used to show that $\u \to
\uhat$ is an isomorphism. To prove this we need the notion of an induced module.
(This is sometimes called a co-induced module, cf.\ \cite[p.~67]{brown}.)

For an $R$-module $V$, define the representation induced from $V$ to $\Rhat$ by
$$
\Ind_R^\Rhat V = \Fun_R(\Rhat,V),
$$
where $\Fun_R$ denotes the set of left $R$-invariant functions $\phi : \Rhat \to
V$. This is a left $\Rhat$-module with respect to the action $(r\phi)(x) =
\phi(xr)$, where $r,x\in \Rhat$. Since $R$ has finite index in $\Rhat$, the
induced representation is a rational representation of $\Rhat$ whenever $V$ is a
rational representation of $R$.

For all $R$-modules $U$ and $\Rhat$-modules $V$, there is a natural
isomorphism
$$
\Hom_R(\Res_R^\Rhat U, V) \cong \Hom_\Rhat(U,\Ind_R^\Rhat V),
$$
where $\Res_R^\Rhat U$ denotes the restriction of $U$ to $R$.

Likewise, for any $\G$ module $V$, we can define $\Ind_\G^\Ghat V =
\Fun_\G(\Ghat,V)$. If $V$ is an $R$-module, viewed as a $\G$-module via
$\rho_\G$, then the restriction mapping
$$
\Ind_R^\Rhat V \overset{\simeq}{\longrightarrow} \Ind_\G^\Ghat V, 
$$
is an isomorphism of $\Ghat$-modules.

To apply Stallings' criterion, we need to show that $H^k(\uhat) \to H^k(\u)$ is
an isomorphism (resp., injection) when $k=1$ (resp., $k=2$). Since $R$ is
reductive, it suffices to show that the natural mapping
$$
\phi_k : \Hom_R(\Res_R^\Rhat H_k(\uhat),V) \to \Hom_R(H_k(\u),V)
$$
is an isomorphism (resp., injection) for all finite dimensional $R$-modules $V$
when $k=1$ (resp., $k=2$). Consider the commutative diagram
{\small
$$
\xymatrix{
\Hom_R(\Res_R^\Rhat H_k(\uhat),V) \ar[r]^\simeq\ar[d]_{\phi_k} &
\Hom_\Rhat(H_k(\uhat),\Ind_R^\Rhat V) \ar[r] &
H^k(\Ghat;\Ind_\G^\Ghat V) \ar@{=}[d] \cr
\Hom_R(H_k(\u),V) \ar[rr] & & H^k(\G;V)
}
$$
}
The right hand vertical map is an isomorphism by Shapiro's Lemma
\cite[p.~73]{brown}. We apply Theorem~\ref{thm:presentation}. When $k=1$, all
horizontal mappings are isomorphisms, which implies that $\phi_1$ is an
isomorphism. When $k=2$, all horizontal mappings are injective, which implies
that $\phi_2$ is injective.
\end{proof}

\begin{example}
Suppose that $n\ge 1$. Denote the subgroup of $\GL_n(R)$ that consists of
matrices with determinant $\pm 1$ by $\SLhat(R)$. Then
Proposition~\ref{prop:fte_ext} implies that the commutative diagram
$$
\xymatrix{
1 \ar[r] & \Aut^+ F_n \ar[d]_{\rhohat^+} \ar[r] & \Aut F_n \ar[d]_{\rhohat}
\ar[r] & C_2 \ar@{=}[d] \ar[r] & 1 \cr
1 \ar[r] & \A_n \ar[d] \ar[r]^\phi & \Ahat_n \ar[d]\ar[r]
& C_2 \ar@{=}[d] \ar[r] & 1\cr
1 \ar[r] & \SL_n(\Q) \ar[r] & \SLhat_n(\Q) \ar[r] & C_2 \ar[r] & 1
}
$$
has exact rows. It follows that $\A_n$ is the identity component of  $\Ahat_n$.
There is a similar story for $\Out F_n$. It is for this reason that in
Section~\ref{sec:aut_out} we considered only the completions of $\Aut^+ F_n$ and
$\Out^+F_n$.
\end{example}

\section{Mapping Class Groups and their Completions}

Suppose that $g,n,r$ are non-negative integers. A decorated surface of type
$(g,n,r)$ is a pair $(S,D)$ where $S$ is a compact oriented surface of genus $g$
and $D=P\cup V$ is a set of decorations, where $P = \{x_1,\dots,x_n\}$ is a set
of $n$ points of $S$ and $V = \{v_1,\dots,v_r\}$ is a set of $r$ non-zero
tangent vectors of $S$. If $v_j \in T_{y_j}S$, the points
$x_1,\dots,x_n,y_1,\dots,y_r$ are required to be distinct. The decorated surface
$(S,D)$ is {\em stable} if the punctured surface $S_D' :=
S-\{x_1,\dots,x_n,y_1,\dots,y_r\}$ has negative Euler characteristic:
$$
\chi(S_D') = \chi(S) - |P| - |V| = 2-2g - (r+n) < 0.
$$

The mapping class group $\Ghat_{S,D}$ of a stable decorated surface $(S,D)$ is
the group of connected components of the group of orientation preserving
diffeomorphisms of $S$ that fix $P$ and $V$ set wise. There is a natural
surjection
$$
\Ghat_{S,D} \to \Aut D := \Aut P \times \Aut V.
$$
For a subgroup $G$ of $\Aut D$ define $\G^G_{S,D}$ to be the inverse image of
$G$ under this homomorphism. The classical mapping group of $(S,D)$ corresponds
to the trivial group $\triv$:
$$
\G_{S,D} := \G_{S,D}^\triv = \pi_0 \Diff^+(S,D).
$$
There is a natural extension
$$
1 \to \G_{S,D} \to \G^G_{S,D} \to G \to 1.
$$

The classification of surfaces implies that $\G_{S,D}^G$ depends only on
$(g,n,r)$ and the subgroup $G$ of $S_n\times S_r$.

For a commutative ring $R$, set $H_R = H_1(S;R)$. The group of automorphisms of
$H$ that preserve the intersection pairing is an algebraic group over $\Q$ that
we shall denote by $\Sp(H)$. There is a natural surjective homomorphism
$$
\rho : \G^G_{S,D} \to G\times \Sp(H_\Z).
$$
Its kernel is, by definition, the Torelli group $T_{S,D}$.

\subsection{Boundary Components versus Tangent Vectors}

Tangent vectors are essentially interchangeable with marked boundary components.
Because boundary components are less natural in algebraic geometry, we prefer to
work with tangent vectors. A marked boundary component of a surface is a
boundary component of the surface together with a point on the boundary
component. Marked boundary components can be exchanged with tangent vectors as
follows:

If $v \in T_y S$ is a non-zero tangent vector of a surface $S$, then one can
replace $(S,v)$ by a surface $\hat{S}$ with a marked boundary component. Here
$\hat{S}$ is the real oriented blowup of $S$ at $y$. This is the surface
obtained from $S$ by replacing $y$ by the circle of rays in $T_yS$. The marked
point on the boundary of $\hat{S}$ corresponds to the ray $\R^+v$ in $T_y
S$ determined by $v$. It will be denoted by $[v]$.

This process may be reversed by collapsing the boundary component to a point $y$
and choosing any non-zero tangent vector at $y$ that lies in the ray in $T_y S$
determined by the marked point. These identifications are well defined and
mutually inverse up to isotopy.

The  corresponding mapping class groups are isomorphic. For example, if $S$ is
compact, then the natural homomorphisms
$$
\xymatrix{
\pi_0 \Diff^+(S,v) \ar[r]^\simeq & \pi_0\Diff^+(\hat{S},[v]) &
\ar[l]_\simeq \pi_0\Diff^+(\hat{S},\partial\hat{S})
}
$$
are isomorphisms.

\subsection{Completions of Mapping Class Groups}
The ground field $F$ will be $\Q$ unless otherwise stated. Suppose that $(S,D)$
is a stable decorated surface and that $G$ is a subgroup of $\Aut D$, where $D =
P\cup V$. The group $G \times \Sp(H)$ is a reductive algebraic group over $\Q$
and the representation $\rho : \G^G_{S,D} \to G\times \Sp(H)$ is Zariski dense.
Denote the completion of $\G^G_{S,D}$ relative to $\rho$ by $\cG^G_{S,D}$. It is
an extension
$$
1 \to \U^G_{S,D} \to \cG^G_{S,D} \to G\times \Sp(H) \to 1.
$$

The next result follows directly from Proposition~\ref{prop:fte_ext}.

\begin{proposition}
For all subgroups $G$ of $\Aut D$, the sequence
$$
1 \to \cG_{S,D} \to \cG_{S,D}^G \to G \to 1
$$
is exact.
\end{proposition}

The proposition implies that $\cG_{S,D}$ is the connected component of the
identity of $\cG^G_{S,D}$.

\begin{corollary}
For all subgroups $G$ of $\Aut D$, the Lie algebra of $\cG^G_{S,D}$ is
$\g_{S,D}$.
\end{corollary}

\begin{theorem}
\label{thm:central}
If $(S,D)$ is a stable decorated surface where $g(S) \ge 3$, then
$$
0 \to \Q \to T_{S,D}^\un \to \cG_{S,D} \to \Sp(H) \to 1
$$
is exact. When $g=2$, the homomorphism $T_{S,D}^\un \to \U_{S,D}$ is
surjective.
\end{theorem}

This result deserves some comment.
Corollary~\ref{cor:surjectivity} implies that $T_{S,D}^\un \to \U_{S,D}$ is
surjective when $g\ge 2$. Theorem~\ref{thm:kernel} implies that
$$
\Q \to T_{S,D}^\un \to \cG_{S,D} \to \Sp(H) \to 1
$$
is exact when $g\ge 3$. The injectivity of $\Q \to T_{S,D}^\un$ is  equivalent
to the non-vanishing of a Chern class. A clumsy proof of the  non-vanishing is
given in \cite{hain:comp}. However, the non-vanishing follows directly from an
earlier result of Morita \cite{morita}, as explained in \cite{hain-reed}.

\subsection{Tautological Homomorphisms}

Suppose that $(S,D)$ is a decorated surface. A decoration $\Dtilde = \Ptilde
\cup \Vtilde$ of $S$ is a {\em refinement} of $D$ if
$$
D\subseteq \Dtilde,\
P \subseteq \Ptilde \cup \Vtilde \text{ and }V \subseteq \Vtilde,
$$
where $D = P\cup V$ and $\Dtilde = \Ptilde \cup \Vtilde$. Thus, in passing
from $\Dtilde$ to $D$, tangent vectors can become points, and points and
tangent vectors can be forgotten.

Suppose that $(S,D)$ is stable. This implies that $(S,\Dtilde)$ is also stable.
For each $G \subseteq \Aut D \cap \Aut \Dtilde$, there is a natural homomorphism
$\G_{S,\Dtilde}^G \to \G_{S,D}^G$.

\subsection{Natural Actions}
\label{sec:actions}

The natural actions of mapping class groups on the fundamental groups of
associated surfaces can be completed.

Suppose that $(S,D)$ is a stable decorated surface where $D=P\cup V$. Recall
that $S_D'$ is the surface obtained from $S$ by removing the support of $D$.

\begin{definition}
\label{def:admissible}
An {\em admissible base point} $x$ of $S_D'$ is either (1) a point $x$ of
$S_D'$ or (2) a tangent vector $x \in V$. Let $\Dtilde = D\cup\{x\}$. This
equals $D$ when $x\in V$.
\end{definition}

If $x$ is an admissible base point of $S_D'$, then $\pi_1(S_D',x)$ is defined.

Suppose that $G$ is a subgroup of $\Aut \Dtilde$ that fixes $x$. It can also be
viewed as a subgroup of $\Aut D$. Denote the Lie algebra of $\pi_1(S_D',x)^\un$
by $\p(S_D',x)$. There are natural actions
$$
\thetatilde_x : \G_{S,\Dtilde}^G \to \Aut \p(S_D',x)
\text{ and }
\theta : \G_{S,D}^G \to \Out \p(S_D').
$$ 

The Zariski closure of the image of each of these is an extension of
$G\times\Sp(H)$ by a prounipotent group. The universal mapping property of
relative completion implies that $\thetatilde_x$ and $\theta$ induce
homomorphisms
$$
\phitilde_x : \cG_{S,\Dtilde}^G \to \Aut \p(S_D',x)
\text{ and }
\phi : \cG_{S,D}^G \to \Out \p(S_D').
$$
These, in turn, induce Lie algebra homomorphisms
$$
d\phitilde_x : \g_{S,\Dtilde} \to \Der \p(S_D',x)
\text{ and }
d\phi : \g_{S,D} \to \OutDer \p(S_D').
$$

\begin{proposition}
If $D$ is non-empty, then $\rhohat : \G_{S,D}^G \to \cG_{S,D}^G$ is injective.
\end{proposition}

\begin{proof}
It suffices to prove that $T_{S,D}$ injects into $\U_{S,D}$. It also suffices to
prove the case where $D$ consists only of points. Write $D = D'\cup \{x\}$. Set
$S' = S - D'$ and $\pi = \pi_1(S',x)$. Then the natural homomorphism $\G_{S,D}
\to \Aut\pi$ is injective. Since $\pi$ is resdidually torsion free nilpotent,
$\pi \to \pi^\un$ is injective. It follows that $\G_{S,D} \to \Aut \p$ is
injective, where $\p$ is the Lie algebra of $\pi^\un$. The result follows as
this homomorphism factors $\G_{S,D} \to \cG_{S,D} \to \Aut \p$, which forces
$\G_{S,D} \to \cG_{S,D}$ to be injective.
\end{proof}

Denote the Lie algebra of $T_{S,D}^\un$ by $\t_{S,D}$.
Since the natural representations $T_{S,x}^\un \to \Aut \p(S,x)$ and $T_S^\un
\to \Out \p(S)$ factor through $\cG_{S,x} \to \Aut \p(S,x)$ and $\cG_S \to \Out
\p(S)$, Theorem~\ref{thm:central} implies:

\begin{proposition}
\label{prop:tor_ker}
When $g\ge 3$, the natural representations $T_{S,x}^\un \to \Aut \p(S,x)$ and
$T_S^\un \to \Out \p(S)$ have non-trivial kernel. Equivalently, both $\t_{S,x}
\to \Der \p(S,x)$ and $\t_S \to \OutDer\p(S)$ have non-trivial kernel. $\Box$
\end{proposition}

When $g\ge 3$, the only known elements of the kernel of $\t_{S,x} \to
\Der\p(S,x)$ are those in $\ker\{\t_{S,x} \to \u_{S,x}\}$. So it is natural (and
interesting) to ask whether $\u_{S,x} \to \Der \p(S,x)$ is injective when $g\ge
3$. (This homomorphism fails to be injective when $g=1,2$.)

\section{Weight Filtrations on Homology and Cohomology}

The rational cohomology\footnote{Recall Convention~\ref{convention}: all
(co)homology is with rational coefficients unless otherwise noted.} of a complex
algebraic variety $X$ carries a natural filtration
\begin{multline*}
0 = W_0 H^m(X) \subseteq W_1 H^m(X) \subseteq \cdots 
\cr
\cdots \subseteq W_{2m-1}H^m(X)
\subseteq W_{2m} H^m(X) = H^m(X)
\end{multline*}
called the {\em weight filtration}, which was constructed by Deligne using Hodge
theory in \cite{deligne:hodge2,deligne:hodge3}. Weight filtrations can be
constructed Galois actions as well \cite{deligne:weil2}. Algebraic maps between
complex algebraic varieties induce weight filtration preserving maps of their
cohomology \cite{deligne:hodge3}. In particular, $(H^\dot(X),W_\dot)$ is a
filtered algebra. The weight filtration is a powerful tool for studying the
topology of complex algebraic varieties due to its strong exactness properties.
In this section we give a brief introduction to weight filtrations directed at
topologists.  Deligne's paper \cite{deligne:icm} provides a more complete
exposition of the yoga of weight filtrations. Full details can be found in
\cite{deligne:hodge2}.

An integer $k$ is a (non-trivial) weight of $H^m(X)$ if its $k$th weight graded
quotient
$$
\Gr^W_k H^m(X) := W_kH^m(X)/W_{k-1}H^m(X)
$$
is non-zero. We say that $H^m(X)$ is {\em pure} of weight $k$ if $k$ is the only
non-trivial weight of $H^m(X)$. The weights on $H^m(X)$ are $\ge m$ when $X$ is
smooth and $\le m$ when $X$ is compact. So if $X$ is smooth and projective, then
$H^m(X)$ is pure of weight $m$.

Since we are working with fundamental groups, it is more natural to work with
weight filtrations on homology than on cohomology. The weight filtration on
$H_m(X)$ is defined by
$$
W_{-k} H_m(X) = \Hom(H^m(X)/W_{k-1},\Q).
$$
When $X$ is smooth, the weights on $H_m(X)$ lie in $\{-2m,\dots,-m\}$.

\begin{example}
\label{ex:wt-filt}
The weight filtration on the homology of a smooth complex algebraic curve is
determined by the topology of the underlying surface. Suppose that $S$ is a
compact oriented surface and that $D$ is a finite subset. Set $S' = S-D$. Then
one has the exact sequence (the dual of the Gysin sequence):
$$
0 \to \Htilde_0(D) \to H_1(S') \to H_1(S) \to 0.
$$ 
The weight filtration on $H_1(S')$ is given by
$$
W_{-k} H_1(S') =
\begin{cases}
0 & k \ge 3 \cr
\Htilde_0(D) & k = 2 \cr
H_1(S') & k \le 1.
\end{cases}
$$
Note that $\Gr^W_{-1}H_1(S') = H_1(S)$. The weight filtration on $H_0(S')$
and $H_2(S')$ are uninteresting.
\end{example}

Higher dimensional examples with non-trivial weight filtrations can be
constructed by taking products of curves. The weight filtration on the product
of two varieties is the tensor product of the two weight filtrations:
$$
W_k H^n(X\times Y) = \bigoplus_{\ell + m = n} \sum_{i+j=k}
W_i H^\ell(X)\otimes W_jH^m(Y).
$$
This induces an isomorphism
$$
\Gr^W_k H^n(X\times Y) \cong \bigoplus_{\ell + m = n} \sum_{i+j=k}
Gr^W_i H^\ell(X)\otimes \Gr^W_jH^m(Y).
$$

\subsection{Strictness and Exactness Properties}
Weight filtrations that arise from Hodge and/or Galois theory have strong
exactness properties which make them a powerful tool in studying the topology
of algebraic varieties and algebraic maps.

\begin{definition}
A morphism $f (V_1,W_\dot) \to (V_2,W_\dot)$ of filtered vector spaces is {\em
strict} with respect to $W_\dot$ if for all $m\in \Z$
$$
W_m V_2 \cap f(V_1) = f(W_m V_1).
$$
\end{definition}

Suppose that $V$ is a vector space and that $A$ is a subspace
of $V$ and $q : V \to B$ is a quotient. A filtration $W_\dot$ of $V$ induces
one on $A$ and $B$ by restriction and projection, respectively:
$$
W_m A := A \cap W_m V \text{ and } W_m B = q(W_m V).
$$
In particular, we can induce filtrations on the kernel and cokernel of a
filtration preserving map $f : (V_1,W_\dot) \to (V_2, W_\dot)$.

It is easy to check that $f$ is strict with respect to $W_\dot$ if and only
if
$$
0 \to \Gr^W_m \ker f \to \Gr^W_m V_1 \to \Gr^W_m V_2 \to \Gr^W_m \coker f \to 0
$$
is exact for all $m \in \Z$.

Another important property of weight filtrations on cohomology groups of
algebraic varieties, established in \cite{deligne:hodge2}, is that there are
natural (but not canonical) isomorphisms
$$
H^m(X;\C) \cong \bigoplus \Gr^W_k H^m(X;\C)
$$
that are preserved by algebraic maps and which are compatible with tensor
products, cup products, the K\"unneth isomorphism, etc. Establishing the
existence of natural splittings of the weight filtration is the essential
ingredient in establishing the strictness and exactness properties stated
above.

Many other invariants of algebraic varieties and maps (such as the Leray
spectral sequence, Gysin sequences, long exact sequences of a pair) carry
natural weight filtrations, and all of their internal maps (differentials, Gysin
maps, connecting homomorphisms) and all maps induced between them by algebraic
maps preserve the weight filtration (sometimes with a shift) and are strict. The
following example of Deligne \cite{deligne:hodge3} illustrates the basic yoga of
weights and how it can be used to prove a topological result.

\begin{example}[Deligne]
Suppose that $G$ is a connected linear algebraic group over $\C$ and that $X$ is
a smooth complex projective variety. Suppose that $\mu : G\times X \to X$ is an
algebraic action. Deligne \cite{deligne:hodge3} shows that the weights on
$H^k(G)$ are strictly larger than $k$ except when $k=0$. Since $X$ is smooth and
projective, $H^k(X)$ is pure of weight $k$. The mapping
$$
\mu^\ast : H^n(X) \to H^n(G\times X) \cong
\bigoplus_{\ell + m = n} H^\ell(G)\otimes H^m(X)
$$
is thus filtration preserving. Since $H^n(X) = W_n H^n(X)$, strictness implies
that
$$
\im \mu^\ast = \im \mu^\ast \cap W_n H^n(G \times X) = H^0(G) \otimes H^n(X)
$$
from which it follows that $\mu^\ast$ is the inclusion
$$
H^n(X) \cong H^0(G) \otimes H^n(X) \hookrightarrow H^n(G\times X).
$$
That is, rational cohomology cannot distinguish $\mu$ from the trivial action.
\end{example}

\section{Weight Filtrations on Completed Mapping Class Groups}

Completions of mapping class groups have natural weight filtrations that are
preserved by the natural homomorphisms between them. They arise because mapping
class groups occur as fundamental groups of smooth stacks (moduli spaces of
curves) and are constructed using either Hodge theory \cite{hain:torelli} or
Galois theory \cite{hain-matsumoto:mcg}.

Denote the lower central series of a Lie algebra $\h$ by
$$
\h = L^1\h \supseteq L^2 \h \supseteq \h \supseteq \cdots
$$
where $L^{m+1} \h := [\h,L^m\h]$.

\begin{theorem}[Hain \cite{hain:torelli}]
If $(S,D)$ is a stable decorated surface and $G$ is a subgroup of $\Aut D$, then
$\O(\cG^G_{S,D})$ has a natural weight filtration with which the product,
antipode and coproduct are strictly compatible. This corresponds to a weight
filtration
$$
\cdots \subseteq W_{-2}\cG^G_{S,D} \subseteq W_{-1}\cG^G_{S,D} \subseteq
W_0 \cG^G_{S,D} = \cG^G_{S,D}
$$
by subgroups, where $W_{-1}\cG^G_{S,D} = \U_{S,D}$. It also induces a filtration
of the Lie algebra $\g_{S,D}$ of the identity component. It has the property
that $\g_{S,D} = W_0 \g_{S,D}$ and $\u_{S,D} = W_{-1}\g_{S,D}$. The adjoint
action
$$
\g_{S,D} \otimes \O(\cG^G_{S,D}) \to \O(\cG^G_{S,D}),
$$
the bracket $\g_{S,D} \otimes \g_{S,D} \to \g_{S,D}$ and the natural
homomorphisms $\g_{S,\Dtilde} \to \g_{S,D}$ are all strictly compatible with the
weight filtration. When $g \ge 3$ and $\#D = 1$, the weight filtration is
related to the lower central series of $\u_{S,D}$ by
$$
W_{-m}\g_{S,D} = L^m\u_{S,D}.
$$
When $g = 0$ and $m\ge 1$,
$$
W_{-2m+1}\g_{S,D} = W_{-2m}\g_{S,D} = L^m\u_{S,D}.
$$
In particular, when $g=0$, $\g_{S,D} = W_{-2}\g_{S,D}$ and all odd weight graded
quotients of $\g_{S,D}$ are trivial.
\end{theorem}

For each subgroup $G$ of $\Aut D$, conjugation induces infinitesimal actions
$$
\ad : \g_{S,D} \to \Der \O(\cG^G_{S,D}) \text{ and }
\ad : \g_{S,D} \to \Der\g_{S,D}.
$$
Since $\Gr^W_0 \g_{S,D} = \g_{S,D}/\u_{S,D} \cong \sp(H)$, we have:

\begin{corollary}
\label{cor:symplectic}
Each $\Gr^W_m\O(\cG^G_{S,D})$ is a direct sum of finite dimensional
$\sp(H)$-modules and each $\Gr^W_m \g_{S,D}$ is a direct product of finite
dimensional $\sp(H)$-modules. $\Box$
\end{corollary}

These weight filtrations are compatible with those constructed (in
\cite{morgan,hain:dht}) on fundamental groups of algebraic curves and their
configuration spaces:

\begin{theorem}[Morgan, Hain]
If $(S,D)$ is a stable decorated surface, then the Lie algebra $\p$ of the
unipotent completion of the fundamental group of the configuration space of $m$
ordered points in $S_D'$ has a natural weight filtration that satisfies $\p =
W_{-1}\p$. In particular, $\p(S_D',x)$ has a natural weight filtration that
satisfies $\p(S_D',x) = W_{-1}\p(S_D',x)$.  The bracket and the surjection
$\p(S_D',x) \to H_1(S_D')$ are strictly compatible with the weight filtration.
When $\#D \le 1$, the weight filtration of $\p(S_D')$ is given by its lower
central series:
$$
W_{-m}\p(S_D',x) = L^m \p(S_D',x)
$$
when $m\ge 1$.
\end{theorem}

The natural action of the $\g_{S,D}$ on the $\p(S_D')$ is compatible with these
weight filtrations.

\begin{theorem}[Hain \cite{hain:torelli}]
If $(S,D)$ is a stable decorated surface and $x$ is an admissible base point of
$S_D'$, then the natural homomorphisms
$$
\g_{S,D\cup\{x\}} \to \Der \p(S_D',x) \text{ and }\g_{S,D} \to \OutDer\p(S_D')
$$
are strictly compatible with the natural weight filtrations.
\end{theorem}

For all stable decorated surfaces $(S,D)$, the weight filtrations on
$$
\g_{S,D},\ \p(S_D',x),\ \Der \p(S_D',x),\ \OutDer\p(S_D')
$$
all have natural splittings.\footnote{In Hodge theory, one usually tensors with
$\C$ first to construct these splittings. However, the machinery of tannakian
categories implies the existence of such splittings over $\Q$. Cf.\
\cite{deligne:hodge2,morgan}} That is, if $\g$ is such a Lie algebra, then there
is a natural isomorphism of complete Lie algebras
$$
\g \cong \prod_m \Gr^W_m \g,
$$
and if $\phi : \g \to \h$ is a natural homomorphism 
between two such Lie algebras, then the diagram
$$
\xymatrix{
\g \ar[r]^(0.3)\simeq\ar[d]_\phi & \prod_m \Gr^W_m \g \ar[d]^{\Gr^W_\dot \phi}
\cr
\h \ar[r]^(0.3)\simeq & \prod_m \Gr^W_m \h
}
$$
commutes.\footnote{Examples of morphisms $\phi : \g \to \h$ that are strictly
compatible with the weight filtration are those which are induced by morphisms
of moduli spaces of curves or are associated with monodromy representations of
fundamental groups of moduli spaces of curves  associated to natural local
systems over moduli spaces such as those associated to families of unipotent
completions of fundamental groups of universal curves and other tautological
bundles.}

The existence of natural splittings allows one to study, {\em without loss of
information}, the infinitesimal actions
$$
d\phitilde_x : \u_{S,\Dtilde} \to \Der \p(S_D',x)
\text{ and }
d\phi : \u_{S,D} \to \OutDer \p(S_D')
$$
using the associated graded actions
$$
\Gr^W_\dot\u_{S,\Dtilde} \to \Der \Gr^W_\dot \p(S_D',x)
\text{ and }
\Gr^W_\dot \u_{S,D} \to \OutDer \Gr^W_\dot\p(S_D').
$$
It also allows us to construct presentations of $\u_{S,D}$ by giving
presentations to their associated weight graded quotients as was done in
\cite{hain:torelli} for the $\u_S$ when $g\ge 6$.

\begin{remark}
One might hope that there are natural weight filtrations on the Lie algebras
$\f_n$, $\a_n$ and $\o_n$ associated to $\Aut^+F_n$ and $\Out^+F_n$ with respect
to which the natural actions $\a_n \to \Der \f_n$ and $\o_n \to \OutDer \f_n$
are strict and for which each Lie algebra is naturally isomorphic to its
associated graded. This would simplify the problem of finding presentations of
$\ia_n$ and $\io_n$.

Such weight filtrations would probably exist if $\Aut F_n$ or $\Out F_n$ were
the fundamental group of an algebraic variety or stack defined over a number
field, or if one could construct an action of the Galois group of (say)
$\Q(\bmu_n)$, where $\bmu_n$ denotes the $n$th roots of unity, on their
profinite completions that was compatible with the action of the Galois group on
the profinite completion of $\pi_1(\bA^1 - \bmu_n,0)$. It is not clear how to
proceed, or if this could ever be true.
\end{remark}

\section{The Relative Weight Filtration of a Nilpotent Endomorphism}

This section is an exposition of the linear algebra of nilpotent endomorphisms
of filtered vector spaces, which arises naturally in the study of degenerations
of complex algebraic varieties. For example, suppose that
$$
f : X \to \Delta
$$
is a family of complex algebraic varieties over the unit disk that is
topologically locally trivial over the punctured disk $\Delta^\ast$. Denote the
fiber of $f$ over $t \in \Delta$ by $X_t$. Fix a base point $t_o \in
\Delta^\ast$. Since the family is locally topologically trivial over
$\Delta^\ast$, there is a monodromy operator\footnote{Recall
Convention~\ref{convention}: all (co)homology is with rational coefficients
unless otherwise noted.}
$$
h : H^m(X_{t_o}) \to H^m(X_{t_o})
$$
for each $m \in \N$. A general result of Griffiths-Landman-Grothendieck (Cf.\
\cite{katz,landman}) implies that the eigenvalues of $h$ are roots of unity. So,
by replacing the family by its pullback along a finite covering $\Delta \to
\Delta, s \mapsto s^c$ if necessary, we may assume that $h$ is unipotent (i.e.,
all of its eigenvalues are 1). In this case it is the exponential of a nilpotent
matrix $N = \log h$.

\begin{example}
\label{ex:classical}
A classical and relevant example occurs when the fiber $X_t$ over $t \in
\Delta^\ast$ is a compact Riemann surface and the central fiber $X_0$ is
obtained from $S = X_{t_o}$ by contracting a a finite set of disjoint simple
closed curves (the {\em vanishing cycles}) $\{c_1,c_2,\dots,c_r\}$ in $S$. The
geometric monodromy $\tau$ is the product of the Dehn twists about the $c_j$.
The induced mapping
$$
h : H_1(S) \to H_1(S)
$$
is given by the {\em Picard-Lefschetz formula}:
$$
h(x) = x + \sum_{j=1}^r \langle c_j,x \rangle c_j.
$$
This is clearly unipotent. Its logarithm $N : H_1(S) \to H_1(S)$ is
the operator 
$$
N : x \mapsto \sum_{j=1}^r \langle c_j,x \rangle c_j
$$
which satisfies $N^2 = 0$. (This formula is independent of the orientations
assigned to the $c_j$.) $\Box$
\end{example}

\subsection{The Weight Filtration of a Nilpotent Endomorphism}

There is a natural weight filtration of a vector space associated to a nilpotent
endomorphism $N$ of it.

\begin{proposition}
\label{prop:wt-filt}
If $N$ is a nilpotent endomorphism of a finite dimensional vector space $V$,
then there is a unique filtration
\begin{multline*}
0 = W(N)_{-m-1} \subseteq W(N)_{-m} \subseteq W(N)_{-m+1} \subseteq \cdots
\cr
\dots \subseteq W(N)_{m-1} \subseteq W(N)_m = V
\end{multline*}
of $V$ such that
\begin{enumerate}

\item for all $n\in \Z$, $NW(N)_n \subseteq W(N)_{n-2}$;

\item for each $k\in \Z$, 
$$
N^k : \Gr_k^{W(N)} V \to \Gr_{-k}^{W(N)} V
$$
is an isomorphism.

\end{enumerate}
\end{proposition}

The filtration $W(N)_\dot$ of $V$ is called the {\em weight filtration of $N$}.

\begin{proof} To prove existence , it is enough to consider the case where $N$
has a single Jordan block. There is a basis
$$
e_{-m}, e_{-m+2}, e_{-m+4}, \dots, e_{m-2}, e_m
$$
of $V$ such that $N e_j = e_{j-2}$. Define
$$
W(N)_j = \Span\{e_k : k \le j\}.
$$

Uniqueness is proved by induction on the exponent of nilpotency of $N$. If
$N=0$, then uniqueness is clear. Suppose that $m>0$ and that $N^{m+1} = 0$, but
that $N^m \neq 0$. The vanishing of $N^{m+1}$ implies that
$$
W(N)_k = V \text{ and } W(N)_{-k-1} = 0 \text{ for all } k \ge m.
$$
Since $N^m : \Gr_m V \to \Gr_{-m}V$ is an isomorphism, we must have
$$
W(N)_{m-1} = \ker N^m \text{ and } W(N)_{-m} = \im N^m.
$$
Since $N^m \neq 0$,
$$
0 \neq  W(N)_{-m} \subseteq W(N)_{m-1} \neq V.
$$
The induced endomorphism $\bar{N}$ of $V' := W(N)_{m-1}/W(N)_{-m}$ satisfies
$\bar{N}^m = 0$. By induction, the weight filtration of $W(\bar{N})$ of
$\bar{N}$ is unique. All weight filtrations of $V$ must satisfy
$$
W(N)_k = \text{ inverse image of } W(\bar{N})_k \text{ whenever } -m < k < m.
$$
Uniqueness follows.
\end{proof}

Note that $W(N)_\dot$ is centered at $0$.
If $V = H^m(X)$, where $X$ is a smooth projective variety, it is natural
to reindex the weight filtration of a nilpotent endomorphism $N$ of $V$ so that
it is centered at the weight $m$ of $V$. The shifted filtration
$$
M_k V := W(N)_{k-m}
$$
is centered at $m$. The reindexed filtration $M_\dot$ satisfies $N M_k \subseteq
M_{k-2}$ and
$$
\xymatrix{
N^k : \Gr_{m+k}^M V \ar[r]^\simeq & \Gr_{m-k}^M V
}
$$
is an isomorphism for all $k \in \Z$. We will call the shifted weight filtration
$M_\dot$ the {monodromy weight filtration} of $N:V\to V$.

\begin{example}
\label{ex:classical_ctd}
The monodromy weight filtration for nilpotent endomorphism $N$ of the weight
$-1$ vector space $H_1(S)$ in Example~\ref{ex:classical} is:
\begin{align}
M_0 &= W(N)_1 = H_1(S) \cr
M_{-1} &= W(N)_0 = \ker N \cr
M_{-2} &= W(N)_{-1} = \im N = \Span\{c_1,\dots,c_r\}\cr
M_{-3} &= W(N)_{-2} = 0.
\end{align}
\end{example}

The existence of a weight filtration extends to arbitrary direct sums and direct
products of nilpotent $N$-modules.

\subsection{Curve Systems}
\label{sec:curve_sys}
A {\em curve system} on a stable decorated surface $(S,D)$ is a set
$$
\gamma = \{c_0,\dots,c_r\}
$$
of disjoint simple closed curves such that each connected component of
$$
S_D' - |\gamma| := S'_D - \bigcup_{j=1}^r c_j
$$
has negative Euler characteristic. Equivalently, no two $c_j$ are isotopic in
$S_D'$ and no $c_j$ bounds a disk or punctured disk in $S'_D$. Two curve systems
are considered to be equal if they are isotopic in $S$.

Denote the Dehn twist about $c_j$ by $\tau_j$. Since the $c_j$ are disjoint,
these commute. Set $\tau = \prod_j \tau_j$. The action
$$
\tau_\ast : H_1(S_D') \to H_1(S_D')
$$
of $\tau$ on $H_1(S_D')$ is given by the {\em Picard-Lefschetz formula}:
$$
\tau_\ast(x) = x + \sum_{j=0}^r \langle c_j,x \rangle c_j.
$$
This is clearly unipotent. Its logarithm $N_\gamma : H_1(S_D') \to H_1(S_D')$ is
the operator $\tau_\ast - \id$ which is given by
$$
N_\gamma : x \mapsto \sum_{j=1}^r \langle c_j,x \rangle c_j.
$$
It satisfies $N_\gamma^2 = 0$. Note that it preserves the weight filtration
$W_\dot$ defined on $H_1(S_D')$ in Example~\ref{ex:wt-filt} and acts trivially
on $W_{-2}H_1(S_D')$.

The following example will be used later in the paper.

\begin{example}
\label{ex:wt_filt}
Suppose that $H = H_1(S)$ and $N = N_\gamma$, where $H$ has weight $-1$. Denote
the corresponding monodromy weight filtration by $M_\dot$. Set
$$
A = \Gr^M_0 H,\ H_0 = \Gr^M_{-1} H,\ B = \Gr^M_{-2} H.
$$
There is a natural isomorphism $\sp(H) \cong S^2 H$, which we consider to have
weight 0 as it is a subspace of $\End(H)$, which has weight 0. Note that
$$
\Gr^M_\dot \sp(H) = \sp(\Gr^M_\dot H)
$$
and that there is a natural Lie algebra isomorphism
$$
\Gr^M_0 \sp(H) \cong \gl(A) \oplus \sp(H_0).
$$
Denote by $\xi$ the element of $\sp(\Gr^M_\dot H)$ that corresponds to the
identity element of $\gl(A)$. Note that $\xi$ acts as the identity on $A$,
minus the identity on $B$ and trivially on $H_0$. It follows that if we consider
$H^{\otimes n}$ to have weight $-n$, then $\xi$ acts on $\Gr^M_k H^{\otimes n}$
as multiplication by $k-n$. It follows that if $V$ is an $\sp(H)$-submodule of
$H^{\otimes n}$, then $\xi$ acts on $\Gr^M_k V$ as multiplication by $k-n$.
\end{example}

\subsection{The Weight Filtration of a Nilpotent Endomorphism of a Filtered
Vector Space}

Now suppose that $N$ is a nilpotent endomorphism of a filtered finite
dimensional vector space $V$. That is, $V$ has a filtration
$$
0 \subseteq \dots \subseteq 
W_{m-1}V \subseteq W_m V \subseteq W_{m+1}V \subseteq \dots \subseteq V
$$
which is stable under $N$. This is extended to the infinite dimensional case
using the conventions of Section~\ref{sec:inf_dim}. Namely, infinite dimensional
examples are either ind- or pro-objects of the category of finite dimensional
filtered vector spaces; the nilpotent endomorphism $N$ is replaced by a locally
nilpotent endomorphism (i.e., a direct limit of nilpotent endomorphisms) in the
ind case and a pronilpotent endomorphism (i.e., an inverse limit of nilpotent
endomorphisms) in the pro case.

We will often call the filtration $W_\dot$ of $V$ the {\em weight filtration} of
$V$ and $\Gr^W_m V$ the $m$th weight graded quotient of $V$.

Natural examples of a filtered vector space $(V,W_\dot)$ with a nilpotent
endomorphism arise from degenerations of smooth (not necessarily compact)
varieties. In this case $(V,W_\dot)$ is $H^m(X_t)$ endowed with its natural
weight filtration, and $N$ is the logarithm of the unipotent part of the
monodromy operator.

Since $N$ preserves the weight filtration, it induces an endomorphism
$$
N_m := \Gr^W_m N : \Gr^W_m V \to \Gr^W_m V.
$$
of the $m$th weight graded quotient of $V$. Since, by assumption, $\Gr^W_m V$ is
the product or sum of nilpotent $N$-modules, Proposition~\ref{prop:wt-filt}
implies that each graded quotient has a weight filtration $W(N_m)$. The
reindexed filtration $W(N_m)[m]_\dot$ is centered at $m$. Denote it by
$M^{(m)}_\dot$

\begin{definition}
A filtration $M_\dot$ of $V$ is called a {\em relative weight filtration} of $N
: (V,W_\dot) \to (V,W_\dot)$  if
\begin{enumerate}

\item for each $k \in \Z$, $NM_k \subseteq M_{k-2}$;

\item the filtration induced by $M_\dot$ on $\Gr^W_m V$ is the reindexed weight
filtration $M^{(m)}_\dot$.

\end{enumerate}
\end{definition}

Relative weight filtrations, if they exist, are unique. (Cf.\
\cite{steenbrink-zucker}).

\begin{example}
\label{ex:trivial}
If $N : (V,W_\dot) \to (V,W_\dot)$ satisfies $N(W_mV)\subseteq W_{m-2}V$ for all
$m\in \Z$, then each $N_m = 0$ and the relative weight flirtation $M_\dot$ of
$N$ exists and equals the original weight filtration $W_\dot$.
\end{example}

\begin{example}
Suppose that $\gamma$ is a curve system on a stable decorated surface $(S,D)$.
Take $V=H_1(S_D')$ with the weight filtration defined in
Example~\ref{ex:wt-filt} and $N$ to be the nilpotent endomorphism $N_\gamma$
associated to $\gamma$ defined in Section~\ref{sec:curve_sys}. The non-trivial
weight graded quotients of $H_1(S_D')$ are
$$
\Gr^W_{-1}H_1(S_D') \cong H_1(S) \text{ and }
\Gr^W_{-2}H_1(S_D') \cong \Htilde_0(D).
$$
Note that $N_{-1}: H_1(S) \to H_1(S)$ is the operator given in
Example~\ref{ex:classical}. Consequently
$M^{(-1)}_\dot$ is given by Example~\ref{ex:classical_ctd}:
$$
M^{(-1)}_{-2}\! H_1(S) = \im N_{-1}, M^{(-1)}_{-1}\! H_1(S) = \ker N_{-1},
M^{(-1)}_{0}\! H_1(S) = H_1(S).
$$
Since $N_{-2} = 0$,
$$
0 = M^{(-2)}_{-3} \subseteq M^{(-2)}_{-2} = \Htilde_0(D).
$$
The relative weight filtration of $N_\gamma : H_1(S_D') \to H_1(S_D')$ exists.
It is defined by
$$
M_{-3} =0, M_{-2} = \im N_\gamma + \Htilde_0(D),
M_{-1} = \ker N_\gamma + \Htilde_0(D),
M_{0} = H_1(S_D').
$$
\end{example}

Even though the weight filtration of a nilpotent endomorphism of a finite
dimensional vector space always exists, the relative weight filtration of a
nilpotent endomorphism of a {\em filtered} vector space $(V,W_\dot)$ usually
does not exist. Necessary and sufficient conditions for the existence of
a relative weight filtration are given in \cite{steenbrink-zucker}.

The existence of a relative weight filtration on the rational cohomology of the
general fiber of a degeneration of complex algebraic varieties was first
established for degenerations of varieties by Deligne
\cite[(1.8)]{deligne:weil2} using $\ell$-adic methods, and for smooth varieties
over the complex numbers using Hodge theory by Steenbrink and Zucker
\cite{steenbrink-zucker}. It provides non-trivial restrictions on the possible
monodromy operators of degenerations of algebraic varieties. For example, the
existence relative weight filtration is a strong enough invariant to show that a
bounding pair (BP) map cannot be the geometric monodromy of a degeneration of
complex algebraic curves:\footnote{This can be proved by elementary and direct
arguments. However, using the non-existence of the relative weight filtration to
prove that a BP map cannot be the geometric monodromy of degeneration of curves
illustrates the kinds of restrictions that the existence of relative weight
filtrations places on the monodromy of degenerations of varieties in general.}

\begin{example}
Suppose that $g(S) \ge 1$ and that the curve system $\{c_0,c_1\}$ is a bounding
pair of simple closed curves in $S$. (That is, $S-|\gamma|$ has two connected
components.) Suppose that $P=\{x_0,x_1\}$ is a pair of points in $S-|\gamma|$,
one in each component. Denote the Dehn twist about $c_j$ by $\tau_j$. The
associated bounding pair map is $\tau = \tau_1 \tau_0^{-1}$. It acts
non-trivially and unipotently on $H_1(S_D')$. Its logarithm $N = \tau_\ast -
\id$ acts trivially on both weight graded quotients of $H_1(S_D')$. Because of
this, the relative weight filtration $M_\dot$, if it exists, must agree with the
weight filtration $W_\dot$. But since these satisfy
\begin{align*}
M_{-3} H_1(S_D') &= W_{-3} H_1(S_D') = 0 \text{ and }\cr
M_{-1} H_1(S_D') &= W_{-1} H_1(S_D') = H_1(S_D'),
\end{align*}
the condition $NM_{-1} \subseteq M_{-3}$ implies that $N=0$. But this
contradicts the non-triviality of $\tau_\ast$. Consequently, the endomorphism
$N$ of  $(H_1(S_D'),W_\dot)$ has no relative weight filtration.
\end{example}

\section{Relative Weight Filtrations on Mapping Class Groups}

An element $\sigma$ of a proalgebraic group $\cG$ is {\em prounipotent} if it
lies in a prounipotent subgroup. An element $N$ of a pro-Lie algebra $\g$ is
{\em pronilpotent} if it lies in a pronilpotent subalgebra.

\begin{lemma}
Each prounipotent element of a proalgebraic group $\cG$ can be written uniquely
as the exponential of a pronilpotent element of $\g$, the Lie algebra of $\cG$.
\end{lemma}

\begin{proof}
Suppose that $\tau$ is a prounipotent element of $\cG$. The existence of a
pronilpotent logarithm of $\tau$ is clear as it lies in a prounipotent subgroup.
Since every algebraic subgroup of a prounipotent group is prounipotent, the
intersection of two prounipotent subgroups of $\cG$ is also prounipotent. If
$\tau = \exp N_1  = \exp N_2$, then lies in the intersection of the two
unipotent 1-parameter subgroups $\{\exp tN_j : t \in F\}$, $j=1,2$. If $\tau
\neq 1$, this forces $N_1=N_2$. If $\tau = 1$, the unique logarithm is $N=0$.
\end{proof}

\begin{proposition}
\label{prop:log_dehn}
If $\tau \in \G_{S,D}^G$ is a Dehn twist, then $\rhohat(\tau)$ is a prounipotent
element of $\cG_{S,D}$.
\end{proposition} 

\begin{proof}
The Picard-Lefschetz formula implies that $\rho(\tau)$ is a unipotent element of
$\Sp(H_1(S))$ and that it is the exponential of $N = \rho(\tau) - \id$. The
inverse image $\H$ of the unipotent subgroup $L := \{\exp tN : t\in \Q\}$ of
$\Sp(H)$ in $\cG_{S,D}$ is prounipotent as it is an extension
$$
1 \to \U_{S,D} \to \H \to L \to 1
$$
of a unipotent group by a prounipotent group. Since $\rhohat(\tau) \in \H$,
it is prounipotent and thus has a unique pronilpotent logarithm.
\end{proof}

Suppose that $(S,D)$ is a stable decorated surface and that $\gamma = \{
c_0,\dots,c_m\}$ is a curve system on $(S,D)$. Denote the Dehn twist on $c_j$ by
$\tau_j$. By Proposition~\ref{prop:log_dehn} $\rhohat(\tau_j)$ has a canonical
logarithm $N_j \in \g_{S,D}$.

Define the closed cone in $\g_{S,D}$ associated to $\gamma$ by
$$
C(\gamma) =
\big\{\sum_{j=0}^m r_j N_j : r_j \in \Q \text{ and } r_j \ge 0\big\}
$$
and the open cone by 
$$
C^o(\gamma) =
\big\{\sum_{j=0}^m r_j N_j : r_j \in \Q \text{ and } r_j > 0\big\}
$$
Then $C(\gamma)$ is a simplicial cone in $\Q^\gamma$ whose faces correspond to
the subsets $\sigma$ of $\gamma$:
$$
C(\gamma) = \coprod_{\sigma \subseteq\gamma} C^o(\sigma).
$$

Suppose that $N \in C(\gamma)$ and that $G$ is a subgroup of $\Aut D$. The
infinitesimal actions
$$
\ad : \g_{S,D} \to \Der \O(\cG^G_{S,D}) \text{ and }
\ad : \g_{S,D} \to \Der\g_{S,D}.
$$
induce actions of $N$ on each weight graded quotient of $\O(\cG^G_{S,D})$ and
$\g_{S,D}$. By Corollary~\ref{cor:symplectic}, each weight graded quotient of
$\O(\cG^G_{S,D})$ is a direct sum of finite dimensional $\sp(H)$-modules and
each weight graded quotient of $\g_{S,D}$ is a direct product of finite
dimensional $\sp(H)$-modules. Consequently, each weight graded quotient of
$\O(\cG^G_{S,D})$ is a direct sum of finite dimensional nilpotent $N$-modules
and and each weight graded quotient of $\g_{S,D}$ is a direct product of finite
dimensional nilpotent $N$-modules.

The following theorem is a special case of more general results proved in
\cite{hain-matsumoto:mcg} using Galois theory and in \cite{hain-pearlstein}
using Hodge theory.

\begin{theorem}
\label{thm:wt_mcg}
For all curve systems $\gamma$ of a stable decorated surface $(S,D)$, all
subgroups $G$ of $\Aut D$, and all $N\in C^o(\gamma)$, there is a (necessarily
unique) relative weight filtration $M_\dot^\gamma$ (denoted $M_\dot$) of
$\O(\cG_{S,D})$ of the endomorphism
$$
\ad(N) \in \Der(\O(\cG^G_{S,D}),W_\dot)
$$
that is compatible with the product, antipode and coproduct of $\O(\cG_{S,D})$.
It induces relative weight filtrations on $\g_{S,D}$ and $\p(S_D',x)$, where $x
\in S_D'-|\gamma|$ is an admissible base point. The bracket of each of these Lie
algebras is strictly compatible with $M_\dot$. These relative weight filtrations
depends only on $\gamma$ and will be denoted by $M^\gamma_\dot$. Each $N_j\in
C(\gamma)$ lies in $M_{-2}\g_{S,D}$. Moreover, for each curve system $\gamma$
there is a natural (though not canonical) isomorphism
$$
\g_{S,D} \cong \prod_{k,m} \Gr^M_k \Gr^W_m \g_{S,D}
$$
of completed Lie algebras. In addition, if $\Dtilde$ is a refinement of $D$ that
contains the base point $x$ and $\gammatilde$ is a curve system on $(S,\Dtilde)$
whose image in $(S,D)$ is $\gamma$, then the natural actions
$$
d\phitilde_x : \g_{S,\Dtilde} \to \Der \p(S_D',x)
\text{ and }
d\phi : \g_{S,D} \to \OutDer \p(S_D').
$$
are strictly compatible with $W_\dot$ and the relative weight filtrations
$M_\dot^\gamma$ and $M_\dot^\gammatilde$. The filtrations $M_\dot$ and $W_\dot$
can be simultaneously split. That is, they can be chosen so that the diagram
$$
\xymatrix{
\g_{S,\Dtilde}\ar[r]^(0.35)\simeq\ar[d] &
\prod_{k,m}\Gr^M_k\Gr^W_m \g_{S,\Dtilde} \ar[d]\cr
\Der \p(S_D',x) \ar[r]^(0.375)\simeq &  \Der \Gr^M_\dot\Gr^W_\dot \p(S_D',x)
}
$$
commutes. There is a similar diagram for $d\phi : \g_{S,D} \to \OutDer
\p(S_D')$.
\end{theorem}

Since the diagonal $\Delta:\O(\cG^G_{S,D})\to
\O(\cG^G_{S,D})\otimes\O(\cG^G_{S,D})$ preserves $M_\dot$, the image of $M_{-1}
\O(\cG^G_{S,D})$ under the diagonal is contained in
$$
M_{-1}\O(\cG^G_{S,D})\otimes \O(\cG^G_{S,D})
+ \O(\cG^G_{S,D})\otimes M_{-1}\O(\cG^G_{S,D}).
$$
This implies that $M_{-1}\O(\cG^G_{S,D})$ is a Hopf ideal of $\O(\cG^G_{S,D})$.
Define
$$
M_0\cG_{S,D}^G = \Spec \big(\O(\cG^G_{S,D})/M_{-1}\O(\cG^G_{S,D})\big).
$$
This is a subgroup of $\cG^G_{S,D}$. Since $M_\dot$ is preserved by the bracket,
$M_k\,\g_{S,D}$ is a pronilpotent Lie subalgebra of $\g_{S,D}$ whenever $k\le
0$. The Lie algebra of $M_0\cG_{S,D}^G$ is $M_0\g_{S,D}$. When $k<0$,
$M_k\g_{S,D}$ is pronilpotent. Denote the corresponding prounipotent subgroup
of $\cG_{S,D}^G$ by $M_k\cG_{S,D}$.

The uniqueness of relative weight filtrations is a strong condition which
implies that $M_\dot^\gamma$ has many nice properties. Some are established in
the following paragraphs.

\begin{proposition}
Suppose that $\gamma$ is a curve system on a stable decorated surface $(S,D)$.
If $\phi \in \G_{S,D}^G$, then the relative weight filtrations of
$\O(\cG^G_{S,D})$ and $\g_{S,D}$ satisfy $M_\dot^{\phi(\gamma)} = \Ad(\phi)
M_\dot^\gamma$.
\end{proposition}

\begin{proof}
For a curve system $\sigma$, let $\tau_\sigma = \prod_{c\in\sigma} \tau_c$
where $\tau_c$ denotes the Dehn twist about $c$. Denote the logarithm of
$\tau_\sigma$ by $N_\sigma$. Since $\tau_{\phi(c)} = \phi\tau_c\phi^{-1}$, it
follows that $ \tau_{\phi(\gamma)} = \phi\tau_\gamma\phi^{-1}$, which implies
$N_{\phi(\gamma)} =  \Ad(\phi)N_\gamma$. Since $N_\gamma(M_k^\gamma) \subseteq
M_{k-2}^\gamma$, it follows that
\begin{multline*}
N_{\phi(\gamma)}\big(\Ad(\phi)M_k^\gamma\big)
= \big(\Ad(\phi)N_\gamma\big) \big(\Ad(\phi)M_k^\gamma\big)
\cr
= \Ad(\phi)\big(N_\gamma \big(M_k^\gamma\big)\big)
\subseteq \Ad(\phi)M_{k-2}^\gamma.
\end{multline*}
The result now follows from the uniqueness of the relative weight filtration.
\end{proof}

The subalgebra $M^\gamma_0\,\g_{S,D}$ of $\g_{S,D}$ behaves like a parabolic
subalgebra of a semi-simple Lie algebra:

Parabolic subalgebras of semi-simple and Kac-Moody Lie algebras are self
normalizing, and correspond to boundary strata in the semi-simple case. The
following result suggests that when $g\ge 2$, the subalgebras
$M^\gamma_0\g_{S,D}$ of $\g_{S,D}$ might provide a good notion of parabolic
subalgebra of $\g_{S,D}$.

\begin{proposition}
\label{prop:parabolic}
If $\gamma$ is a curve system on a stable decorated surface $(S,D)$ where
$g(S)\ge 3$, then the normalizer of $M^\gamma_0\g_{S,D}$ in $\g_{S,D}$ is
$M^\gamma_0\g_{S,D}$.
\end{proposition}

\begin{proof}
Since the functor $\Gr^M_\dot \Gr^W_\dot$ is exact, it suffices to prove the
result for the associated bigraded Lie algebra $\Gr^M_\dot\Gr^W_\dot\g_{S,D}$.
Set $M_\dot = M_\dot^\gamma$ and $H=H_1(S)$. Let $\xi \in \Gr^M_0\sp(H)$ be the
element defined in Example~\ref{ex:wt_filt}. It lies in
$\Gr^M_0\Gr^W_0\g_{S,D}$. Johnson's theorem \cite{johnson:h1} implies that 
$\Gr^W_m\g_{S,D}$ is an $\sp(H)$-quotient of $(H^n\oplus\Lambda^3 H)^{\otimes
m}$, where $n=\#D-1$. It follows from Example~\ref{ex:wt_filt} that if $k>0$ and
$X \in \Gr^M_k\Gr^W_m \g_{S,D}$, then $[\xi,X] = (k-m)X$, which is non-zero
whenever $k>0$. Thus, if $X \notin M_0\Gr^M_\dot\Gr^W_\dot \g_{S,D}$, then $X$
does not normalize $M_0\Gr^M_\dot\Gr^W_\dot \g_{S,D}$.
\end{proof}

This result also holds in genus 2, but not in genus 1.

\subsection{Dependence on $\gamma$}

For a curve system $\gamma$ of a stable decorated surface $(S,D)$, denote the
subspace of $H_1(S)$ spanned by the homology classes of the $c\in \gamma$ by
$\langle \gamma \rangle$.

\begin{proposition}
\label{prop:rel_wt}
If $\gamma$ is a curve system of a stable decorated surface $(S,D)$ and $\sigma
\subseteq \gamma$, then the relative weight filtrations $M_\dot^\gamma$ and
$M_\dot^\sigma$ of $\O(\cG^G_{S,D})$, $\g_{S,D}$ and $\p(S_D')$ are equal if and
only if $\langle \sigma \rangle = \langle \gamma \rangle$. 
\end{proposition}

\begin{proof}
The condition that $\langle \gamma \rangle = \langle \sigma \rangle$ implies
that the monodromy weight filtrations on $H_1(S)$ are equal. Since each weight
graded quotient of $\O(\cG^G_{S,D})$, $\g_{S,D}$ and $\p(S_D')$ is a subquotient
of a tensor power of $H_1(S)$, it follows that the monodromy weight filtrations
associated to $\gamma$ and $\sigma$ on $\Gr^W_\dot \g_{S,D}$ are  equal if
and only if $\langle\gamma\rangle = \langle\sigma\rangle$. Similarly for
$\p(S_D')$ and $\O(\cG^G_{S,D})$.

To complete the proof, we now show that if the monodromy weight filtrations on
$\Gr^W_\dot V$ agree, where $V =\O(\cG^G_{S,D})$, $\g_{S,D}$ or $\p(S_D')$, then
the relative weight filtrations on $V$ agree. Since $\sigma \subseteq \gamma$,
for each $c\in \sigma$, $N_c(M_k^\gamma) \subseteq M_{k-2}^\gamma$. That is
$M_\dot^\gamma$ is a relative weight filtration on $V$ for each $N \in
C^o(\sigma)$. Uniqueness implies that $M_\dot^\sigma = M_\dot^\gamma$.
\end{proof}

When $g(S)=0$, $H_1(S)=0$ and the hypotheses of Proposition~\ref{prop:rel_wt}
are satisfied when $\sigma$ is empty. This implies that the relative weight
filtration is the existing weight filtration:

\begin{corollary}
\label{cor:rational}
If $\gamma =\{c_0,\dots,c_m\}$ is a curve system on a stable decorated surface
$(S,D)$ of genus 0 and $G$ is a subgroup of $\Aut D$, then the relative weight
filtrations of $\O(\cG_{S,D}^G)$, $\p(S_D',x)$ and $\g_{S,D}$ equal their
natural weight filtrations $W_\dot$. Consequently, $\cG^G_{S,D} = M^\gamma_0
\cG^G_{S,D}$. $\Box$
\end{corollary}

Suppose that $(S,D)$ is a stable decorated surface, where $D=P\cup V$. For the
purposes of this definition, we will consider $V$ as a set of boundary
components. Following Hatcher-Thurston \cite{hatcher-thurston} we say that two
curve systems $\gamma'$ and $\gamma''$ of $(S,D)$ differ by an {\em $A$-move} if
$$
\{c_1,c_2,c_3,c_4\} \subseteq (\gamma' \cap \gamma'') \cup V
$$
that bound a genus 0 subsurface $T$ of $S$ and there are $c_0'\in \gamma'$ and
$c_0'' \in \gamma''$ that lie in $T$ and
$$
\gamma' = \sigma \cup \{c_0'\} \text{ and } \gamma' = \sigma \cup \{c_0'\}.
$$
Figure~\ref{fig:move} illustrates an $A$-move.

\begin{figure}[!ht]
\epsfig{file=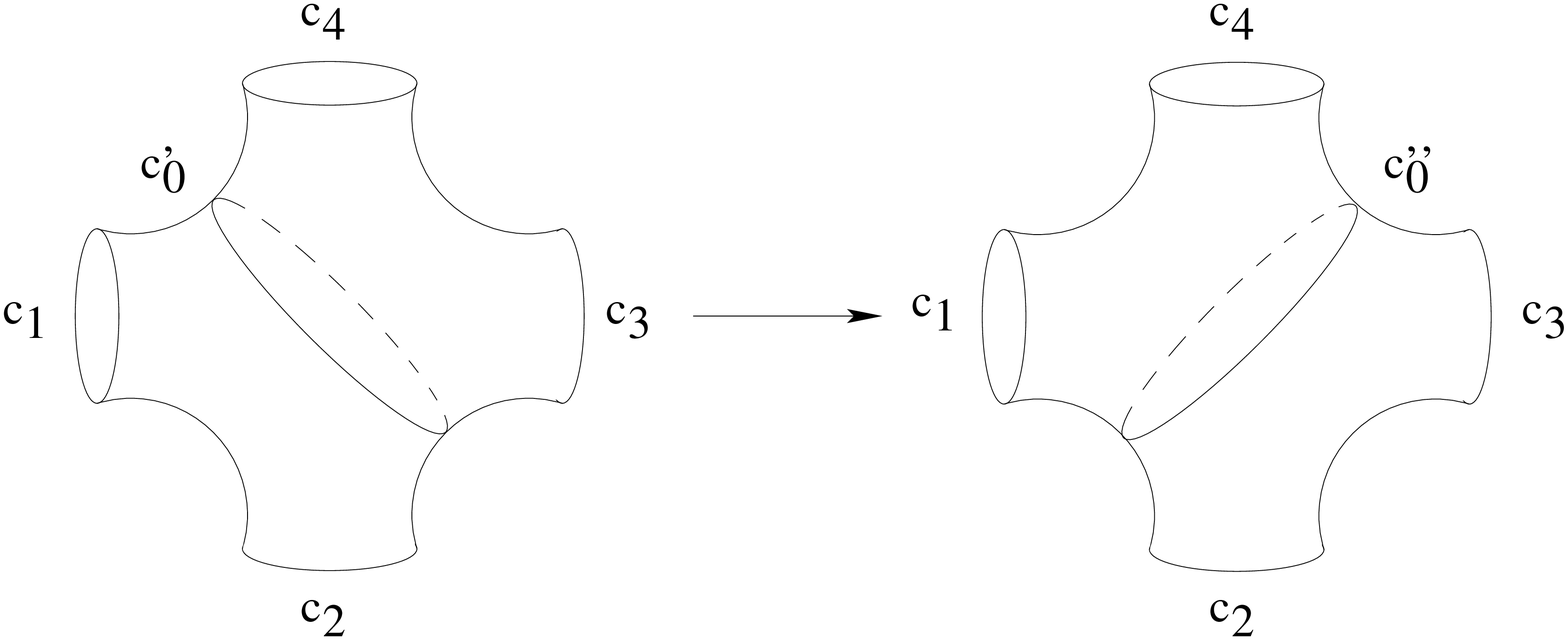, width=2.5in}
\caption{An $A$-move}\label{fig:move}
\end{figure}

Suppose that $\sigma$ and $\gamma$ are curve systems on the stable decorated
surface $(S,D)$. We say that {\em $\gamma$ is obtained from $\sigma$ by homology
neutral insertions} (or that $\sigma$ is obtained from $\gamma$ by homology
neutral deletions) if $\sigma \subseteq \gamma$ and the subspaces
$\langle\sigma\rangle$ and $\langle\gamma\rangle$ of $H_1(S)$ are equal.

The invariance of the $M_\dot^\gamma$ under homology neutral insertions
and deletions follows directly from Proposition~\ref{prop:rel_wt}.

\begin{proposition}
\label{prop:invariance}
Suppose that $(S,D)$ is a stable decorated surface. If $\gamma_1$ and $\gamma_2$
are two curve systems on $(S,D)$ that differ by a sequence of $A$-moves and by
homology neutral insertions and deletions, then the relative weight filtrations
$M_\dot^{\gamma_1}$ and $M_\dot^{\gamma_2}$ of $\g_{S,D}$ are equal. $\Box$
\end{proposition}

\subsection{Glueing Lemma}
\label{sec:glueing}

Suppose that $\gamma = \{c_0,\dots,c_m\}$ is a curve system on the stable
decorated surface $(S,D)$. For each connected component $T'$ of $S'_D -
|\gamma|$ there is a compact oriented surface $T$ and decorations $D_T = P_T
\cup V_T$ such that
$$
P_T = T'\cap P
$$ 
and $V_T$ is the union of $T'\cap V$ with the new tangent vectors obtained by
collapsing the boundary components created by removing the $c_j$.\footnote{It is
convenient and natural to choose a point on each $c_j$, as it will provide a
marking on each boundary component of $S'_D-|\gamma|$.}

There is a natural homomorphism (the {\em glueing map})
\begin{equation}
\label{eqn:curve_homom}
\prod_{T} \G_{T,D_T} \to \G_{S,D}
\end{equation}
whose image is the subgroup of $\G_{S,D}$ that are represented by
diffeomorphisms that restrict to the identity on each $c_j$. The kernel is
isomorphic to $\Z^\gamma$.

\begin{proposition}[Glueing Lemma]
\label{prop:glue}
The glueing map induces a homomorphism on relative completions such that the
diagram
$$
\xymatrix{
\prod_{T} \G_{T,D_T} \ar[d]\ar[r] & \G_{S,D} \ar[d] \cr
\prod_{T} \cG_{T,D_T} \ar[r] & \cG_{S,D}
}
$$
commutes. The induced Lie algebra homomorphism
$$
\bigoplus_T \g_{T,D_T} \to \g_{S,D}
$$
is strict with respect to the weight filtration $W_\dot$ on the $\g_{T,D_T}$
and the relative weight filtration $M_\dot^\gamma$ on $\g_{S,D}$.
\end{proposition}

\begin{proof}
Here we prove the existence of the induced homomorphism $\prod_{T} \cG_{T,D_T}
\to \cG_{S,D}$. Its compatibility with the filtrations follows from the
existence of limit mixed Hodge structures \cite{hain-pearlstein} or,
alternatively, the Galois equivariance of the corresponding map of profinite
mapping class groups.

Set $H=H_1(S)$ and $H_T=H_1(T)$. The relative weight filtration $M_\dot :=
M_\dot^\gamma$ of $H$ satisfies
$$
\Gr^M_{-1} H = \oplus_T H_T \text{ and } M_{-2}H = \langle \gamma \rangle.
$$
Set
$$
M_m \Sp(H) := \{\phi \in \Sp(H) : (\phi-\id)(M_k H)\subseteq M_{k+m} H\}.
$$
This is the subgroup of $\Sp(H)$ with Lie algebra $M_m\sp(H)$. The Zariski
closure of the image of $\prod_{T} \G_{T,D_T}$ in $\Sp(H)$ is contained in $M_0
\Sp(H)$ and is the extension of the subgroup $\prod_T \Sp(H_T)$ of $\Gr^M_0
\Sp(H)$ by a unipotent subgroup of $M_{-2}\Sp(H)$. Denote by $\H$ the inverse
image of $\prod_{T} \G_{T,D_T}$ under the surjection $M_0\cG_{S,D} \to \Gr^M_0
\Sp(H)$. It is an extension of $\prod_{T} \G_{T,D_T}$ by the prounipotent group
$M_{-2}\cG_{S,D}$.

The completion of $\prod_{T} \G_{T,D_T}$ with respect to the natural
homomorphism $\prod_T \G_{T,D_T} \to \prod_T \Sp(H_T)$ is $\prod_{T}
\cG_{T,D_T}$. The universal mapping property of relative completion implies that
the homomorphism $\prod_{T} \G_{T,D_T} \to \H$ induces a homomorphism $\prod_{T}
\cG_{T,D_T} \to \H$ such that the diagram
$$
\xymatrix{
\prod_{T} \G_{T,D_T} \ar[d]\ar[rr] && \G_{S,D} \ar[d] \cr
\prod_{T} \cG_{T,D_T} \ar[r] & \H\ar[r] & \cG_{S,D}
}
$$
commutes.
\end{proof}

There is a more elaborate version of the Glueing Lemma that applies to groups
that contain
$$
\prod_{T} \G_{T,D_T}^{G_T} 
$$
as a finite index subgroup, where $G_T \subseteq \Aut D_T$, and which map to
$\G^G_{S,D}$ for certain $G\subseteq \Aut D$. Rather than formulate such a
result in general, we now state and prove a very special case that we shall need
when investigating handlebody subgroups of $\G_{S,D}$.

Consider the decorated surface $(S,D)$ of genus $h-1$ with one marked boundary
component that is constructed as the double covering of the 2-sphere branched at
the $2h$th roots of unity, $\bmu_{2h}$, and with the disk of radius $1/2$
removed from one of the branches. This surface can be described as the Riemann
surface of the algebraic function
$$
y^2 = x^{2h} - 1
$$
with the disk $|x| < 1/2$ removed from one of the two branches. The marked
boundary point is chosen to be $x=1/2$.

\begin{figure}[!ht]
\epsfig{file=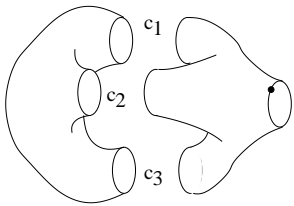, width=1.5in}
\caption{The surface $S-|\gamma|$ when $h=3$}\label{fig:twist}
\end{figure}

For $j=1,\dots,h$, let $c_j$ be the circle in $S$ that is the inverse image in
$S$ of the interval $[\zeta^{2j-1},\zeta^{2j}]$ in $\C$, where $\zeta=\exp{2\pi
i/2h}$. Then $\gamma := \{c_1,\dots,c_h\}$ is a curve system that separates $S$
into two genus $0$ subsurfaces $(T_0,D_0)$ and $(T_1,D_1)$, where
$$
D_0 = \{c_1,\dots,c_h\} \text{ and } D_1 = \{\partial S\} \cup D_0.
$$
The case $h=3$ is illustrated in Figure~\ref{fig:twist}.

The group $\bmu_h$ of $h$th roots of unity acts on $S$: $\zeta^{2j} : (x,y)
\mapsto (\zeta^{2j}x,y)$. It acts on $D_0$ and $D_1$ by taking $c_j$ to
$c_{j+1}$, where the indices are considered mod $h$. Denote the natural
homomorphism $\G^G_{T_j,D_j} \to \bmu_h$ by $p_j$. Set
$$
[\G_{T_0,D_0} \times \G_{T_1,D_1}]^{\bmu_h} =
\{(\phi,\phi_2) \in \G^{\bmu_h}_{T_0,D_0} \times \G^{\bmu_h}_{T_1,D_1} :
p_1(\phi_1) = p_2(\phi_2)\}. 
$$
There is an obvious glueing homomorphism
\begin{equation}
\label{eqn:prod}
[\G_{T_0,D_0} \times \G_{T_1,D_1}]^{\bmu_h} \to \G_{S,D}.
\end{equation}

The completion of $[\G_{T_0,D_0} \times \G_{T_1,D_1}]^{\bmu_h}$ with respect to
the natural homomorphism to $\bmu_h$ is easily seen to be the restriction of
$\cG^{\bmu_h}_{T_0,D_0}\times \cG^{\bmu_h}_{T_1,D_1}$ to the diagonal of $\bmu_h
\times \bmu_h$. Denote it by
$$
[\cG^{\bmu_h}_{T_0,D_0}\times \cG^{\bmu_h}_{T_1,D_1}]^{\bmu_h}.
$$

\begin{proposition}
The homomorphism (\ref{eqn:prod}) induces a homomorphism
$$
[\cG^{\bmu_h}_{T_0,D_0}\times \cG^{\bmu_h}_{T_1,D_1}]^{\bmu_h} \to \cG_{S,D}
$$
that is strictly compatible with the induced mapping
$$
(\O(\cG_{S,D}),M_\dot^\gamma) \to
(\O([\cG^{\bmu_h}_{T_0,D_0}\times \cG^{\bmu_h}_{T_1,D_1}]^{\bmu_h}),W_\dot).
$$
\end{proposition}

The existence of the induced homomorphism is proved by an argument similar to
the proof of Proposition~\ref{prop:glue}. The strictness with respect to the
filtrations follows from either Hodge theory or Galois theory.

Since $T_0$ and $T_1$ are spheres, Corollary~\ref{cor:rational} implies that
$\cG^{\bmu_h}_{T_j,D_j} = W_0 \cG^{\bmu_h}_{T_j,D_j}$. This gives the following
important consequence of strictness.

\begin{corollary}
The image of $[\cG^{\bmu_h}_{T_0,D_0}\times \cG^{\bmu_h}_{T_1,D_1}]^{\bmu_h}
\to \cG_{S,D}$ lies in $M_0 \cG_{S,D}$. $\Box$
\end{corollary}

Define $\phi_h \in \G_{S,D}$ to be the mapping class of the diffeomorphism
$$
(x,y) \mapsto (e^{2\pi i /h}x,y)
$$
of $S$ composed with $1/h$th of the inverse of the Dehn twist about $\partial
T$. This fixes the boundary point $1/2$. Observe that $\phi_h^h$ is the inverse
of the Dehn twist about $\partial T$. Since $\phi_h$ lies in the image of
(\ref{eqn:prod}), we have:

\begin{proposition}
\label{prop:phi}
The image of $\phi_h$ in $\cG_{S,D}$ lies in $M_0\cG_{S,D}$.
\end{proposition}

We shall also need a certain diffeomorphism $\psi$ of a surface $S$ of genus 2
with one boundary component. It is convenient to take $S$ to the Riemann surface
$$
y^2 = (x^2-1)(x^2-4)(x^2-9)
$$
with one of the two preimages of the disk $|x|<1/2$ removed. It is illustrated
in Figure~\ref{fig:psi}.
\begin{figure}[!ht]
\epsfig{file=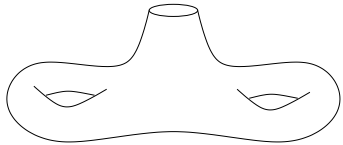, width=1.75in}
\caption{The diffeomorphism $\psi$}\label{fig:psi}
\end{figure}
The element $\psi$ of $\G_{S,\partial S}$ is the composition of the
diffeomorphism $(x,y) \mapsto (-x,y)$ of $S$ with the square root of the inverse
of the Dehn twist about the boundary of $S$. In terms of the illustration in
Figure~\ref{fig:psi}, it is obtained by rotating $S$ by $\pi$ about the vertical
axis of the boundary circle and composing with the square root of the inverse of
the Dehn twist about the boundary of $S$. An argument similar to the one used to
prove Proposition~\ref{prop:phi} can be used to prove:

\begin{proposition}
\label{prop:psi}
The image of $\psi$ in $\cG_{S,\partial S}$ lies in $M_0\cG_{S,\partial S}$.
\end{proposition}

\section{Handlebodies and Relative Weight Filtrations}
\label{sec:handle_filt}

A curve system $\gamma = \{c_0,\dots,c_m\}$ on a stable decorated surface
$(S,D)$ is said to be {\em rational} if each connected component of $S -
|\gamma|$ is a genus 0 surface.\footnote{This is equivalent to the condition
that the nodal surface $S/\gamma$ obtained from $S$ by collapsing each $c\in
\gamma$ to point has the topological type of a stable rational curve. It is also
equivalent to the condition that $\Gr^M_{-1}H_1(S) = 0$.}

\begin{lemma}
If $\gamma$ is a rational curve system on a stable decorated surface $(S,D)$,
then there is a unique handlebody $U_\gamma$ such that $S = \partial U_\gamma$
and each curve $c \in \gamma$ bounds a disk in $U$.\footnote{That is, if $U_1$
and $U_2$ are handlebodies with $\partial U_1 = \partial U_2 = S$ where each
$c\in \gamma$ bounds in each $U_j$, then there is a homeomorphism $f : U_1 \to
U_2$ that is the identity on $S$.}
\end{lemma}

\begin{proof}
This follows directly from the elementary fact that an oriented $2$-sphere
bounds a handlebody (necessarily a $3$-ball) in a unique way, as the handle body
is the cone over $S$.
\end{proof}

A maximal curve system on a stable decorated surface $(S,D)$ is called a {\em
pants decomposition} of $S_D'$. If $\gamma$ is a pants decomposition of $S_D'$,
then each component of $S_D'-|\gamma|$ is a sphere with $r$ boundary components
and $n$ punctures, where $r+n=3$. Thus pants decompositions are rational and
determine a handlebody $U_\gamma$. It is proved in \cite{hatcher}that any two
pants decompositions of $S$ can be joined by $A$-moves and $S$-moves. It is
clear that $A$-moves leave the handlebody $U_\gamma$ unchanged while $S$-moves
change the handlebody.

\begin{figure}[!ht]
\epsfig{file=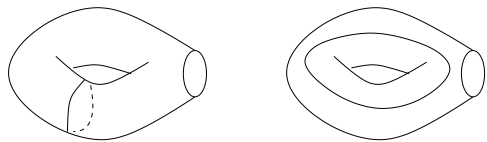, width=2.5in}
\caption{An $S$-move}\label{fig:smove}
\end{figure}

\begin{theorem}
\label{thm:pants}
Two pants decompositions of a decorated stable surface determine the same
handlebody if and only if they can be joined by $A$-moves.
\end{theorem}

\begin{proof}[Sketch of Proof] It is clear that $A$-moves do not change the
handlebody. We use ideas from Morse theory to prove the converse. They are an
elaboration of ideas used by Hatcher and Thurston \cite{hatcher-thurston} (see
also \cite{hatcher}). We will use boundary components instead of tangent vectors
and will assume, without loss of generality, that the decorations $D$ consist
entirely of boundary components.

If $\gamma$ is a pants decomposition of $(S,D)$, then the boundary of the
associated handlebody $U$ equals
$$
\partial U = S \cup \bigcup_{v\in \pi_0(\partial S)} \D_v,
$$
where $\D_v$ is a disk that corresponds to the boundary component $v$ of $S$. We
regard this as a manifold with corners at the submanifold $\partial S$ of
$\partial U$.

To a pants decomposition $\gamma$ of $S$, we associate a graph $G_\gamma$. This
has one white vertex for each $c\in \gamma$ and one black vertex for each pair
of pants. Edges connect a black and a white vertex; each black vertex is joined
to the white vertex corresponding to its boundary components. All black vertices
have valence 3; white vertices have valence 1 if they are boundary components of
$S$, otherwise they have valence 2. The handlebody $U$ is a regular
neighbourhood of $G_\gamma$. The circles retract onto the white vertices.

A height function $f:G_\gamma \to [0,1]$ on such a graph is a continuous
function whose restriction to each edge has no critical points and whose local
extrema occur only at white vertices. We also require that $f$ vanish on each
1-valent white vertex. Every $G_\gamma$ has a height function.

We will say that a smooth function $F : U \to [0,1]$ is {\em Morse} if it
vanishes identically on each $\D_v$, it has no critical points in $U$, and if
its restriction $F|_S$ to $S$ is a Morse function. A Morse function $F:U \to
[0,1]$ is {\em convex} if $F^{-1}(a)$ is a disjoint union of contractible sets
for each $a\in [0,1]$. 

A height function $f$ extends to a convex Morse function $F:U\to [0,1]$ where
$F|_S$ has one critical point for each black vertex and for each 2-valent white
vertex that is a local extremum of $f$. The critical values of $F|_S$ equal the
values of $f$ on the black vertices and the values of $f$ on the 2-valent white
vertices that are local extrema. The stable or unstable manifold in $S$ of the
critical point of $F|_S$ corresponding to a 2-valent white vertex is the isotopy
class of the corresponding $c\in\gamma$. The white vertices that are not local
extrema of $f$ correspond to components of the level sets of $F$. The isotopy
class of the vertex corresponding to $c\in \gamma$ is $c$.

Suppose now that $\gamma_0$ and $\gamma_1$ are two pants decompositions of
$(S,D)$ that determine the same handlebody $U$. Choose height functions $f_j :
\gamma_j \to [0,1]$. Extend these to convex Morse functions $F_j : U \to [0,1]$
using the construction in \cite{hatcher-thurston}. It follows from
\cite{hatcher-thurston} that the result will follow if we can show that there is
a smooth function $F : U\times [0,1] \to [0,1]$ that vanishes identically on
each $\D_v \times [0,1]$, whose restriction to $S \times [0,1]\to [0,1]$ is a
generic 1-parameter family of Morse functions, and where the restriction $F_t :
U \to [0,1]$ to each $U\times\{t\}$ has no critical points and is a convex Morse
functions for all but finitely many $t\in [0,1]$.

To construct such an $F$, first extend $F_0$ and $F_1$ to Morse functions $H_j
: (M,\partial M) \to ([0,1],0)$, where
$$
M = S^3 - \cup_{v\in \pi_0(\partial S)} B^3
$$
and where $\D_v$ is a hemisphere of the boundary of the corresponding 3-ball.
Join these by a generic 1-parameter family of functions $H_t : (M,\partial M)
\to ([0,1],0)$ that have no critical points in a neighbourhood of $\partial M$
and where $H : (M,\partial M)\times [0,1]\to ([0,1],0)\times [0,1]$ that takes
$(x,t)$ to $(H_t(x),t)$ is smooth. The critical set $\Sigma \subset M\times
[0,1]$ of $H$ is 1-dimensional and has relative dimension 0 over $[0,1]$. On the
other hand, we can choose $U$ so that it is a regular neighbourhood of a graph
$\G$ in $M$. Then $U\times [0,1]$ is a regular neighbourhood of $\G\times [0,1]$
in $M\times [0,1]$. Since $\Gamma \times [0,1]$ has relative dimension 1 over
$[0,1]$ and is disjoint from $\Sigma_t:= \Sigma \cap M\times\{t\}$ when $t=0,1$,
there is a vector field on $M\times [0,1]$ that is tangent to the $t$-slices and
whose flow moves $\Gamma \times [0,1]$ in $M\times [0,1]$ to a subset $J$ that
is disjoint from $\Sigma$ and intersects each fiber in a graph $J_t$
homeomorphic to $\G$. Then we may identify a regular neighbourhood $W$ of $J$ in
$M$ with $U\times [0,1]$ where $W_t := W\cap (M\times \{t \})$ is homeomorphic
with $U$ and does not intersect $\Sigma$. We may then deform the restriction of
$H$ to $W$ so that it is a generic 1-parameter family of functions $\partial
S\times [0,1] \subset \partial W$ and has no critical points in the interior of
$W$.

The issue now is that the restriction $H_t : U \to [0,1]$ of $H_t$ to $U \to
[0,1]$ may not be convex as there may be $a,t\in [0,1]$ where $H_t^{-1}(a)$ is
not a disjoint union of contractible sets. There are parameter values
$0<t_1<t_2<\dots<t_m < 1$ where $H_t|_U$ is Morse when $t \notin
\{t_1,\dots,t_m\}$. On each interval $(t_{j-1},t_j)$ of
$[0,1]-\{t_1,\dots,t_m\}$, the Morse function $H_t|_S$ is represented by a graph
$G_j$ and a height function $h_j : G_j \to [0,1]$. As $t$ moves through $t_j$,
two vertices of $G_j$ reverse height and $G_j$ changes into $G_{j+1}$ by one of
the elementary moves described in \cite{hatcher-thurston}. The extra data of the
function $H_t : U \to [0,1]$, when $t \in (t_{j-1},t_j)$, is determined by an
equivalence relation on imbedded intervals in $G_j$ that correspond to
non-critical values of $H_t$: two ``strands'' are equivalent if the
corresponding circles are boundary components of the same component of
$H_j^{-1}(a)$ for some non-critical value $a\in [0,1]$.

\begin{figure}[!ht]
\epsfig{file=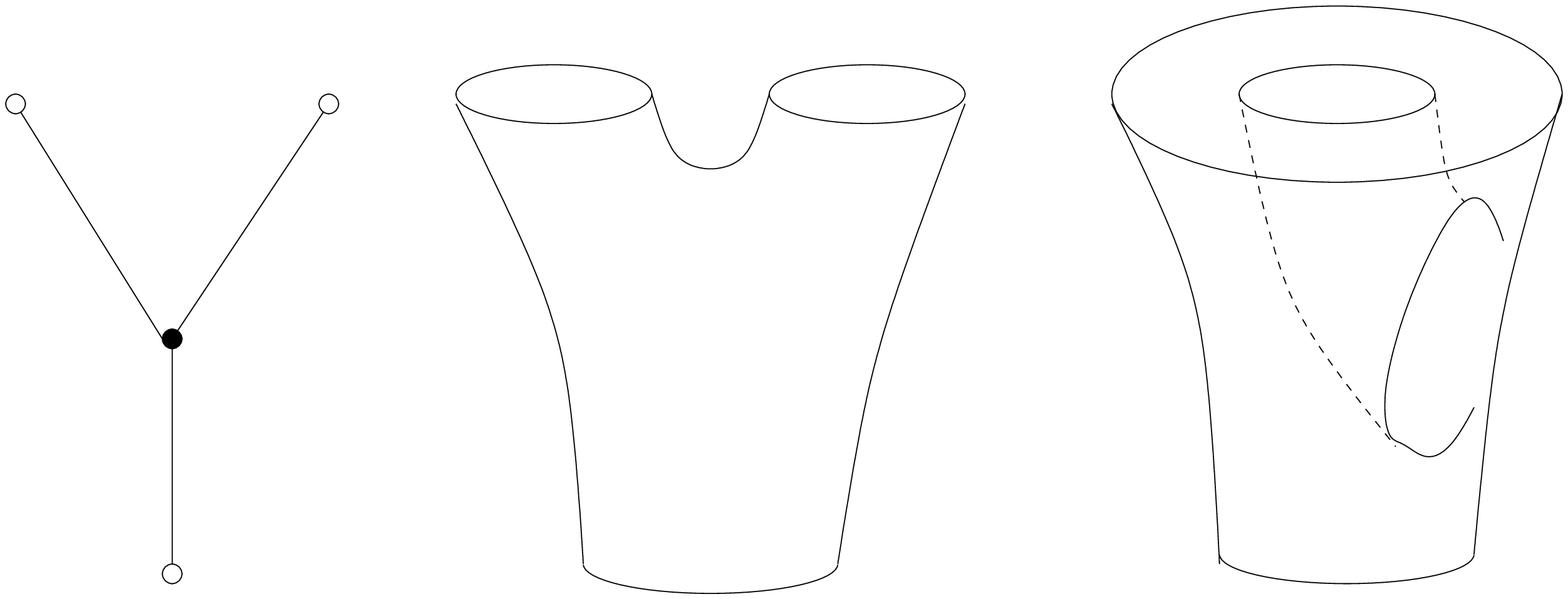, width=2.75in}
\caption{Interior versus Exterior}\label{fig:pants}
\end{figure}

Even if the Morse functions $H_t$ are not convex, we can use the sequence of
graphs $(G_0,h_0),\dots,(G_m,h_m)$ to construct a sequence of $A$-moves that
join $\gamma_0$ to $\gamma_1$. Note that at a black vertex, the level set can
change from, for example, from a disk to to two disk or to an annulus, as
illustrated in Figure~\ref{fig:pants}. However, we can apply the
Hatcher-Thurston construction to each graph to obtain a sequence of  convex
Morse functions $K_j : U \to [0,1]$  where $0\le j \le m$ where $K_0$ and $K_1$
correspond to the pants decompositions $\gamma_0$ and $\gamma_1$. Denote by
$\mu_j$ the pants decomposition of $S$ that corresponds to the restriction of
$K_j$ to $S$. Then it follows from examining the list of elementary moves in
\cite{hatcher-thurston} that the pants decompositions determined by two convex
Morse functions that differ by an elementary move, themselves differ by an
$A$-move. In particular, $\mu_{j-1}$ and $\mu_j$ differ by an $A$-move. This
shows that $\gamma_0 =\mu_0$ and $\gamma_1 = \mu_m$ are connected by a sequence
of $A$-moves.
\end{proof}

This result has the important consequence that each way of writing $S$ as the
boundary of a handlebody $U$ defines a relative weight filtration on invariants
of $(S,D)$, such as $\g_{S,D}$.

\begin{corollary}
\label{cor:M^U}
Suppose that $(S,D)$ is a stable decorated surface and that $x$ is an 
admissible base point of $S_D'$. If $U$ is a handlebody and $S = \partial U$,
then $U$ determines relative weight filtrations $M_\dot^U$ on $\g_{S,D}$,
$\O(\cG_{S,D})$, $\g_{S,D\cup\{x\}}$ and $\p(S_D',x)$. The actions
$$
\g_{S,D} \to \OutDer \p(S_D') \text{ and }
\g_{S,D\cup\{x\}} \to \Der \p_{S_D',x}
$$
are strict with respect to the weight filtrations $W_\dot$ and the relative
weight filtrations $M_\dot^U$.
\end{corollary}

\begin{proof}
Choose a pants decomposition $\gamma$ of $(S,D\cup\{x\})$ such that each $c\in
\gamma$ bounds a disk in $U$. Define $M_\dot^U = M_\dot^\gamma$.
Theorem~\ref{thm:pants} and Proposition~\ref{prop:invariance} implies that this
is independent of the choice of $\gamma$. The strictness properties follow from
Theorem~\ref{thm:wt_mcg}.
\end{proof}

For a handlebody $U$ and $x\in U$, denote $\pi_1(U,x)^\un$ by $\F(U,x)$ and its
Lie algebra by $\f(U,x)$.

\begin{proposition}
\label{prop:Mp}
If $S$ bounds the handlebody $U$ and $x\in S$, then
$$
M_0^U \p(S,x) = \p(S,x) \text{ and }
M_{-2}^U\p(S,x) = \ker\{\p(S,x) \to \f(U,x)\}.
$$
Consequently, $\f(U,x) \cong \Gr^{M^U}_0 \p(S,x)$.
\end{proposition}

\begin{proof}
Choose a pants decomposition $\gamma$ of $(S,x)$ as in the proof of
Corollary~\ref{cor:M^U} such that $M_\dot^\gamma = M^U_\dot$. The relative
weight filtration on $\p(S,x)$ arises from a degeneration of $(S,x)$ to the
stable rational curve $C_0$ whose underlying surface is $(S/\gamma,x)$. The
degeneration can be chosen to be defined over $\Q$ as each of its components is
a 3-pointed $\P^1$.\footnote{We shall denote the degeneration by $X\to \D$,
where $\D$ is an analytic disk in the Hodge case and a formal disk
$\Spec\Q[[t]]$ in the Galois case.}  The inclusion of a nearby fiber $C_t$ (the
fiber over a first order smoothing of $C_0$) into the total space of the local
deformation $X \to \D$ induces the homomorphism $\pi_1(S,x)\cong \pi_1(C_t,x)
\to \pi_1(C_0,x)\cong \pi_1(U,x)$. This implies that the homomorphism $\p(S,x)
\to \f(U,x)$ is equivariant and a morphism of mixed Hodge structures. Since the
Galois action on $H_1(C_0)$ is trivial (resp., the MHS on $\f(U,x)$ is of pure
of weight 0 and type $(0,0)$), it follows that $\f(U,x)$ is a trivial Galois
module (resp., is also pure of weight 0 and type $(0,0)$). Since $\p(S,x) \to
\f(U,x)$ is a morphism (Galois, Hodge), it is strict with respect to the weight
filtration $M_\dot^U$. The assertions follows from strictness and the fact that
$$
M_{-2}^U H_1(\p(S,x)) = \ker\{H_1(\p(S,x)) \to H_1(\f(U,x))\}.
$$
\end{proof}

\section{Handlebody Groups}
\label{sec:handle_gps}

Suppose that $(S,D)$ is a stable decorated surface and that $S$ is the boundary
of a handlebody $U$. Define the {\em handlebody group} of $(U,D)$ by
$$
\Lambda_{U,D} = \pi_0\Diff^+(U,D),
$$
where each diffeomorphism acts trivially on $D$.\footnote{For a subgroup $G$ of
$\Aut D$, one can also define $\Lambda_{U,D}^G$ as we did for mapping class
groups. However, we shall not need such groups.} Griffiths \cite{griffiths},
Suzuki \cite{suzuki}, Luft \cite{luft}, and Pitsch \cite{pitsch} have proved
fundamental results about handlebody groups and found generating sets of
$\Lambda_U$. For example, Griffiths \cite{griffiths} proved that homomorphisms
$$
\Lambda_{U,x} \to \Aut \pi_1(U,x) \text{ and } \Lambda_U \to \Out \pi_1(U)
$$
are surjective. Luft proved that the kernels of each of these is generated by
twists on imbedded disks $(\D,\partial \D) \hookrightarrow (U,S)$.\footnote{The
twist on an imbedded disk is the isotopy class of a smoothing of the
homeomorphism $(re^{i\theta},t) \mapsto (re^{i\big(\theta+2\pi t)},t\big)$ of a
tubular neighbourhood $\D\times [0,1]$ of $\D$ in $U$. Its restriction to $S$ is
the Dehn twist about the loop $\partial \D$ in $S$.}

Restriction to the boundary defines a homomorphism
$$
r_{U,D} : \Lambda_{U,D} \to \G_{S,D}
$$
It is straightforward to show that if $\Dtilde$ is a refinement of $D$, then
$r_{U,\Dtilde}$ and $r_{U,D}$ induce an isomorphism
\begin{equation}
\label{eqn:kernel}
\xymatrix{
r: \ker\{\Lambda_{U,\Dtilde}\to \Lambda_{U,D}\} \ar[r]^{\simeq} &
\ker\{\G_{U,\Dtilde}\to \G_{U,D}\}.
}
\end{equation}
For this reason, we will mainly restrict our attention to the cases where
$\# D\le 1$.

The following appears to be well-known to the experts. I am grateful to Alan
Hatcher for communicating a proof.

\begin{proposition}
\label{prop:folk}
If $(S,D)$ is a stable surface, then the homomorphism $r_{U,D}$ is injective.
\end{proposition}

\begin{proof}[Sketch of Proof]
By the isomorphism (\ref{eqn:kernel}), it suffices to prove the result when $D$
is empty and $g\ge 2$ and when $(g,n)$ is $(0,3)$ and $(1,1)$.

In genus 0 (with any number of points), the result follows directly from a
result of Cerf \cite{cerf}. The general case is proved by induction on $g$.
Suppose $g\ge 1$. Suppose that $\phi \in \Diff^+ U$ is a diffeomorphism whose
restriction to the boundary is isotopic to the identity. Then $\phi$ is isotopic
to a diffeomorphism whose restriction to $S$ is the identity. We will assume
that this is the case. Now choose an imbedded disk $(\D,\partial \D) \subset
(U,S)$. Let $(\D',\partial \D')$ be a disk imbedded in $(U,S)$ parallel to and
disjoint from $\D$. After altering $\phi$ by an isotopy fixing $S$ if necessary,
we may assume that the restriction of $\phi$ to $\D$ is transverse to $\D'$.
Since $U$ is an irreducible 3-manifold \cite{jaco}, $\D' \cup \phi(\D) \cup A$,
where $A \subset S$ is the annulus between $\partial\D$ and $\partial\D'$,
bounds a 3-ball. We can then deform $\phi$ by an isotopy fixing $S$ so that the
number of connected components of the complement in $U$ of $\D'\cup \phi(\D)
\cup A$ is reduced by one. We may therefore assume that $\D'\cup \phi(\D) \cup
A$ bounds a ball. By further modifying $\phi$ by an isotopy fixing $S$, we may
assume that $\phi$ fixes $\D$ pointwise. Now cut $U$ apart along $\D$ to obtain
a diffeomorphism $\phi'$ of a handlebody $U'$ of genus one less whose
restriction to $\partial U'$ is the identity. The result now follows by
induction.
\end{proof}

The handlebody group $\Lambda_U$ is bounded by the 0th term $M_0^U$ of the
relative weight filtration.

\begin{lemma}
\label{lem:M_0}
If $(S,D)$ is a stable decorated surface and $S$ bounds the handlebody $U$, then
the image of $\Lambda_{U,D} \to \cG_{S,D}$ lies in $M^U_0 \cG_{S,D}$.
\end{lemma}

\begin{proof}
The proof would be straightforward if $\u_{S,x} \to \Der \p(S,x)$ were
injective. Since this is not known, we need a direct proof. As noted above, Luft
\cite{luft} proved that $\ker\{\Lambda_{U,x} \to \Aut \pi_1(U,x)\}$ is generated
by twists on imbedded disks. If $c$ is a simple closed curve in $S$ that bounds
an imbedded disk in $U$, then there is a pants decomposition $\gamma$ of $(S,D)$
that contains $c$ where each $c'\in \gamma$ bounds in $U$. It follows from
Theorem~\ref{thm:wt_mcg} that that the Dehn twist on $c$ lies in
$M^U_{-2}\cG_{S,D} := M^\gamma_{-2}\cG_{S,D}$.

Thus, to prove the result, it suffices to show that there are elements of
$M_0\Lambda_{U,x} := \Lambda_{U,x}\cap M^U_0\cG_{S,x}$ whose images in $\Aut
\pi_1(U,x)$ generate $\Aut\pi_1(U,x)$. To do this, we use elements of
$\Lambda_{U,x}$ closely related to those used by Luft in \cite{luft}. That these
generators lie in $M_0^U\cG_{S,x}$ follows from Propositions~\ref{prop:phi}
and \ref{prop:psi}.

Represent $\pi_1(U,x)$ as a free group $\langle a_1,\dots,a_g \rangle$, where
each $a_j$ is a simple closed curve on the boundary of $S$. Note that the
automorphisms conjugate to $\phi_2 \in M_0\Lambda_{U,x}$ constructed in
Section~\ref{sec:glueing} can be used to invert any generators $a_j$ of
$\pi_1(U,x)$ while leaving the remaining generators fixed. The automorphism
$\psi \in M_0\Lambda_{U,x}$ defined there can be used to define an automorphism
that fixes all but two of the generators $a_j$ and $a_k$ and acts on them via
$$
a_j \mapsto a_k^{-1} \text{ and } a_k \mapsto a_j^{-1}.
$$
Composing this with the first kind of automorphism, we see that there are
elements of $M_0\Lambda_{U,x}$ that transpose any two of the generators $a_j$.
We can therefore realize all permutations of the generators $a_j$ by elements of
$M_0\Lambda_{U,x}$. Finally, the elements $\phi_3 \in M_0\Lambda_{U,x}$ realize
the automorphism that fixes $a_j$ when $j>2$ and satisfies
$$
a_1 \mapsto a_2 \text{ and } a_2 \mapsto (a_1 a_2)^{-1}.
$$
By a Theorem of Nielsen \cite{nielsen} (cf.\ \cite{luft}) these automorphisms of
$\pi_1(U,x)$ generate $\Aut \pi_1(U,x)$. This completes the proof.
\end{proof}

When $\#D\ge 1$, the homomorphism $\G_{S,D}\to \cG_{S,D}$ is injective.

\begin{theorem}
\label{thm:M_Lambda}
If $(S,D)$ is a stable decorated surface that bounds the handlebody $U$, where
$\# D = 1$, then
\begin{enumerate}

\item  $\Lambda_{U,D} = \G_{S,D} \cap M_0^U \cG_{S,D}$;

\item
$\ker\{\Lambda_{U,D} \to \Aut \pi_1(U,x)\} =
\Lambda_{U,D} \cap M^U_{-2}\cG_{S,D}$;

\item $\Lambda_{U,D} \cap M^U_{-2}\U_{S,D}$ is generated by opposite twists on
disjoint bounding pairs of imbedded disks $(\D_j,\partial\D_j) \hookrightarrow
(U,S)$ ($j=1,2$) whose images avoid $D$.
\end{enumerate}
\end{theorem}

The second assertion is a consequence of a result of Griffith \cite{griffiths}
and the third a restatement of a result of Pitsch \cite[Prop.~6]{pitsch}.

\begin{proof}
We prove the result when $x$ is a point of $S$. The case where $D$ is a non-zero
tangent vector $v \in T_x S$ follows as $\g_{S,v} \to \g_{S,x}$ is strict with
respect to $M_\dot^U$ and the kernel is central and generated by a Dehn twist on
a curve that bounds a disk in $U$, and therefore lies in $M^U_{-2}$.

Proposition~\ref{prop:Mp} combined with the fact that the homomorphisms
$\pi_1(S,x) \to \cP(S,x)$ and $\pi_1(U,x) \to \F(U,x)$ are injective implies
that the commutative diagram
$$
\xymatrix{
1 \ar[r] & \pi_1(S,x)\cap M_{-2}^U\cP(S,x) \ar[r]\ar[d] & \pi_1(S,x)
\ar[r]\ar[d] & \pi_1(U,x) \ar[r]\ar[d] & 1 \cr
1 \ar[r] & M_{-2}^U\cP(S,x) \ar[r] & \cP(S,x) \ar[r] & \F(U,x) \ar[r] & 1
}
$$
has exact rows. Since $\G_{S,x}$ is a subgroup of $\Aut\pi_1(S,x)$, it follows
that $\G_{S,x}$ is a subgroup of $\Aut\cP(S,x)$. Consequently $\G_{S,x}\cap
M^U_0\Aut\cP(S,x)$ consists of those automorphisms of $\pi_1(S,x)$ that preserve
$\ker\{\pi_1(S,x) \to \pi_1(U,x)\}$. By a result of Griffiths \cite{griffiths}
this is $\Lambda_{U,x}$, so that
\begin{equation}
\label{eqn:incln_0}
\G_{S,x} \cap M^U_0\Aut\cP(S,x) = \Lambda_{U,x}.
\end{equation}
Since $\cP(S,x) = M^U_0\cP(S,x)$, and since $\pi_1(U,x) \to \F(U,x)$ is
injective, the commutativity of the diagram above implies that
\begin{equation}
\label{eqn:incln_2}
\G_{S,x} \cap M^U_{-2}\Aut\cP(S,x) \subseteq
\ker\{\Lambda_{U,x} \to \Aut\pi_1(U,x)\}.
\end{equation}
Since the homomorphism $\g_{S,x} \to \Der\p(S,x)$ preserves the filtration
$M_\dot^U$, it follows that for all $k$ we have
\begin{equation}
\label{eqn:incln_3}
\G_{S,x} \cap M_k^U\cG_{S,x} \subseteq  \G_{S,x} \cap M_k^U\Aut\cP(S,x).
\end{equation}

Lemma~\ref{lem:M_0} implies that $\Lambda_{U,x}\subseteq \G_{S,x}\cap
M^U_0\cG_{S,x}$. The first assertion follows by combining this with the
inclusions (\ref{eqn:incln_0}) and (\ref{eqn:incln_3}) with $k=0$.

By a result of Luft \cite{luft},
the kernel of $\Lambda_{U,x} \to \Aut\pi_1(U,x)$ is generated by Dehn twists on
simple closed curves in $S$ that bound a disk imbedded in $U$. But these lie in
$M_{-2}^U\cG_{S,x}$ by Theorem~\ref{thm:wt_mcg}. Therefore
$$
\ker\{\Lambda_{U,x} \to \Aut\pi_1(U,x)\}
\subseteq \G_{S,x}\cap M^U_{-2}\cG_{S,x}.
$$
The second assertion follows by combining this with the inclusions
(\ref{eqn:incln_2}) and (\ref{eqn:incln_3}) with $k=-2$.

The final assertion follows from this
and a result of Pitsch \cite[Prop.~6]{pitsch}.
\end{proof}

\begin{corollary}
There is a natural injective homomorphism
$$
\Aut \pi_1(U,x) \hookrightarrow \Gr^{M^U}_0\! \cG_{S,D}.
$$
\end{corollary}

This homomorphism induces homomorphisms on the relative completion of
$\Aut^+\pi_1(U,x)$ and $\Out^+\pi_1(U,x)$. Surprisingly, these are not
surjective. Equivalently, the injection in the previous corollary is not
Zariski dense.

\begin{proposition}
If $g\ge 3$, then the induced homomorphisms
$$
\a_n \to \Gr^{M^U}_0\! \g_{S,x} \text{ and } \ia_n \to \Gr^{M^U}_0\! \g_{S}
$$
are not surjective.
\end{proposition}

\begin{proof}[Sketch of Proof]
Denote $H_1(S)$ by $H$ and $H_1(U)$ by $A$. Denote the relative weight
filtration $M_\dot^U$ by $M_\dot$. Denote the kernel of $H\to A$ by $B$. Then $A
= \Gr^M_0 H$ and $B=\Gr^M_{-2}H$. Since $\Gr^W_0\g_{S,x} = \sp(H) \cong S^2H$,
it follows that
$$
\Gr^M_0 \Gr^W_0 \g_{S,x} \cong A\otimes B \cong \End(A) \cong \gl(A).
$$
There are natural $\sp(H)$-equivariant isomorphisms
$$
H_1(\u_{S,x})  \cong \Gr^W_{-1}\g_{S,x} \cong H_1(T_{S,x}) \cong \Lambda^3 H
$$
given by the Johnson homomorphism and general results in \cite{hain:torelli}.
The exactness properties of $\Gr^M_\dot$ and $\Gr^W_\dot$ imply that there
are $\gl(A)$-equi-variant isomorphisms
$$
\Gr^M_0\Gr^W_{-1} \g_{S,x} \cong B\otimes\Lambda^2 A \cong \Hom(A,\L_2(A)),
$$
where $\L_m(A)$ denotes the $m$th graded quotient of the free Lie algebra
generated by $A$. Moreover, the mapping $\IA_n \to \Gr^M_0\U_{S,x}$ induces
Magnus' isomorphism
$$
H_1(\IA_n) \to \Hom(A,\L_2(A)).
$$
By \cite[(10.1),\S11]{hain:torelli}, the second weight graded quotient of
$\g_{S,x}$ is the sum of the $\Sp(H)$-modules that corresponds to the partitions
$[2,2]$ and $[1,1]$. A straightforward linear algebra computation shows that, as
$\gl(A)$-modules,
$$
\Gr^M_0\Gr^W_{-2} \g_{S,x} \cong B\otimes \L_3(A) \cong \Hom(A,\L_3(A)).
$$
Alternatively, it is isomorphic to the kernel of the natural surjection
$S^2\Lambda^2 H \to \Lambda^4 H$ minus a copy of the trivial representation. The
image of $[\ia_n,\ia_n]$ in this group is a quotient of
$$
\Lambda^2 H_1(\IA_n) = \Lambda^2 \Hom(A,\L_2(A)).
$$
Since $S^2 A$ is a summand of $\Hom(A,\L_3(A))$ but not of this group, the
homomorphism $\ia_n \to \Gr^M_0\u_{S,x}/W_{-3}$ is not surjective. The result
follows.
\end{proof}

This result shows that the relative weight filtration of $\cG_{S,x}$ is not
simply obtained by taking the Zariski closure of a filtration of $\G_{S,x}$.

At first glance, this result appears to contradict Theorem~\ref{thm:M_Lambda}.
and the fact that the image of $T_{S,x} \to \U_{S,x}$ is Zariski dense. However,
these results simply say that given $n\ge 1$ and a $\Q$-rational element $\phi$
of $M_0\U_{S,x}/W_{-n}$, there exists $\psi\in T_{S,x}$ and a positive integer
$m$ such that
$$
\phi^m \equiv \psi \bmod W_{-n}\U_{S,x}.
$$
The previous two results imply that when $n>2$, it is not always possible to
choose $\psi$ to lie in $\Lambda_{U,x} = \G_{S,x} \cap M_0\cG_{S,x}$.

Theorem~\ref{thm:M_Lambda} and Proposition~\ref{prop:parabolic} yield the
following strengthening Theorem~\ref{thm:pants}. It says that the different ways
of writing $(S,x)$ as the boundary of a handlebody is faithfully represented in
the set of relative weight filtrations of $\g_{S,x}$.

\begin{corollary}
\label{cor:equiv}
For two pants decompositions $\gamma_1$ and $\gamma_2$ of a stable decorated
surface $(S,D)$, where $\#D = 1$, the following are equivalent:
\begin{enumerate}

\item $\gamma_1$ and $\gamma_2$ are connected by $A$-moves;

\item $U^{\gamma_1} = U^{\gamma_2}$;

\item the associated relative weight filtrations $M_\dot^{\gamma_1}$ and
$M_\dot^{\gamma_2}$ of $\g_{S,D}$ are equal.

\end{enumerate}
\end{corollary}

\begin{proof}
Theorem~\ref{thm:pants} gives the equivalence of (i) and (ii).
Proposition~\ref{prop:invariance} established that (i) implies (iii). It remains
to prove that (iii) implies (ii). We will show that not (ii) implies not (iii).

Set $U_j = U^{\gamma_j}$ and $M_\dot^{\gamma_j} = M_\dot^{U_j}$. There exists
$\phi \in \G_{S,D}$ that extends to a diffeomorphism $\phitilde : U_1 \to U_2$.
Then
$$
\Lambda_{U_2} = \phi \Lambda_{U_1} \phi^{-1} \text{ and } M_\dot^{U_2} =
\Ad(\phi)M_\dot^{U_1}.
$$
If $U_1 \neq U_2$, then $\phi \notin \Lambda_{U_1}$. Theorem~\ref{thm:M_Lambda}
implies that $\phi \notin M_0^{U_1}\cG_{S,D}$. We will prove the result by
showing that $\phi$ does not normalize $M_0^{U_1}\cG_{S,D}$.

Since $\cG_{S,D}$ is connected, it suffices to prove the Lie algebra version: if
$X\in \g_{S,D}$ and $X\notin M^{U_1}_0\g_{S,D}$, then $X$ does not normalize
$M^{U_1}_0\g_{S,D}$. But this follows directly from
Proposition~\ref{prop:parabolic}.
\end{proof}

\section{Extending Diffeomorphisms to Handlebodies}
\label{sec:applications}

In this section we give an application to the problem of bounding the subset of
elements of $\G_{S,D}$ consisting of mapping classes that extend to some
handlebody. Similar results have been obtained independently by Jamie Jorgensen
\cite{jorgensen}.

View $\cG_{S,D}$ as a proalgebraic variety over $\Q$. It is filtered by its
weight filtration
$$
\cG_{S,D} = W_0\cG_{S,D} \supseteq W_{-1}\cG_{S,D} \supseteq W_{-2}\cG_{S,D}
\supseteq \cdots
$$
where $W_{-1}\cG_{S,D} = \U_{S,D}$ and $W_{-m}\cG_{S,D}$ is the $m$th term of
the lower central series of $\U_{S,D}$. Recall from \cite{hain:torelli} that
when $g\ge 3$ and $m\neq 2$, $\Gr^W_{-m}\U_{S,D}$ is isomorphic to the $m$th
graded quotient of the lower central series of the Torelli group $T_{S,D}$
tensored with $\Q$.

Write $S$ as the boundary of a handlebody $U$. Then the set of elements of
$\G_{S,D}$ that extend across some handlebody with boundary $S$ is
$$
C := \bigcup_{\phi \in \G_{S,D}} \phi \Lambda_{U,D} \phi^{-1}.
$$
For all $m\ge 1$, set $C_m = C\cap W_{-m}\cG_{S,D}$. Denote the Zariski closure
of $C_m$ in $\cG_{S,D}$ by $X_m$.

\begin{theorem}
If $(S,D)$ is a stable decorated surface, then $X_m$ is a proper subvariety of
$W_{-m}\cG_{S,D}$ for all
$$
m \ge 
\begin{cases}
4 & \text{ when } g=3; \cr
2 & \text{ when } g =4,5,6; \cr
1 & \text{ when } g \ge 7.
\end{cases}
$$
\end{theorem}

In some sense, this theorem says that most elements of $W_{-m}\G_{S,D}$ do not
extend to any handle body.

Denote the Zariski closure of $\Lambda_{U,D}$ in $\cG_{S,D}$ by $\cL_{U,D}$ and
its intersection with $W_{-m}\U_{S,D}$ by $W_{-m}\cL_{U,D}$. Since $\G_{S,D}$ is
Zariski dense in $\cG_{S,D}$, $C_m$ is contained in the Zariski closure of the
image of the map
$$
F : \cG_{S,D} \times W_{-m}\cL_{U,D} \to \cG_{S,D}
$$
defined by $F(g,\lambda) = g\lambda g^{-1}$.

To prove the result we show that the image of $C_m$ in $\Gr^W_{-m}\cG_{S,D}$ is
contained in a proper subvariety. Since $\Lambda_{U,D}$ is contained in
$M_0^U\cG_{S,D}$, $W_{-m}\cL_{U,D}$ is a subgroup of $M_0^UW_{-m}\cG_{S,D}$.
Consequently, the image of $C_m$ in $\Gr^W_{-m}\cG_{S,D}$ is contained in the
Zariski closure of the image of the map
$$
\Gr^W_0\cG_{S,D} \times M^U_0\Gr^W_{-m}\cG_{S,D} \to \Gr^W_{-m}\cG_{S,D}.
$$
induced by conjugation.

The following is an immediate consequence of a theorem of Chevalley
\cite{chevalley}, which can be found in exercises 3.18 and 3.19 of
\cite[Chap.~II, sect.~3]{hartshorne}.

\begin{lemma}
Suppose that $X$ is a quasi-projective variety over a field and that $Y$ is a
closed subvariety. If $G \times X \to X$ is the action of an algebraic group on
$X$, then the image $G\cdot Y$ of the restricted action $G\times Y \to X$ is a
constructable subset of $X$ whose Zariski closure in $X$ has dimension
$$
\dim \overline{G\cdot Y} = \dim Y + \dim G - \dim G_Y
$$
where $G_Y = \{g\in G: g(Y) \subseteq Y\}$. $\Box$
\end{lemma}

Recall that $\Gr^W_0\cG_{S,D}\cong \Sp(H)$ where $H=H_1(S)$. We apply the Lemma
to the adjoint action of $G= \Sp(H)$ on
$$
X = \Gr^W_{-m}\cG_{S,D} \cong \Gr^W_{-m}\g_{S,D}
\text{ where }Y = M^U_0 \Gr^W_{-m}\g_{S,D}.
$$
Proposition~\ref{prop:parabolic} implies that $G_Y = M^U_0\Sp(H)$. Applying the
Lemma, and using the fact that $\sp(H) = M^U_2\sp(H)$, we see that the
codimension of the closure of the $\Sp(H)$ orbit of $M^U_0\Gr^W_{-m}\g_{S,D}$
in $\Gr^W_{-m}\g_{S,D}$ satisfies
\begin{multline*}
\codim\overline{\Sp(H)\cdot M^U_0\Gr^W_{-m}\g_{S,D}} \cr
= \dim\Gr^W_{-m}\g_{S,D}/M_0^U - \dim \Sp(H)/M_0\Sp(H)\cr
\ge \dim \Gr^M_2\Gr^W_{-m}\g_{S,D} - \dim \Gr^{M}_2\sp(H).
\end{multline*}
It remains to show this is positive for all $m$ in the statement of the theorem.
First, since $\Gr^M_2\sp(H)$ is the symmetric square of a maximal isotropic
subspace of $H$, it has dimension $g(g+1)/2$.

We use representation theory to find a lower bound for the other term. Each
$\Gr^M_k\Gr^W_m\g_{S,D}$ is a $\Gr^M_0\Gr^W_0\g_{S,D}$-module. Recall from the
proof of Corollary~\ref{cor:equiv} that $\Gr^M_0\Gr^W_0\g_{S,D}$ is isomorphic
to $\gl_g$, so that its irreducible representations are given by Young diagrams
with $\le g$ rows. These are the same Young diagrams that parametrize the
irreducible $\sp(H)$-modules, where $H=H_1(S)$. 

\begin{proposition}
If $g\ge 3$ and $m>1$, then $\Gr^M_2\Gr^W_{-m}\g_{S,D}$ contains the
$\gl_g$-module corresponding to the partition $[k,k]$ when $m=2k$ and $[k,k,1]$
when $m=2k-1$.
\end{proposition}

\begin{proof}
Results of Oda \cite{oda} and Asada-Nakamura \cite{asada-nakamura} imply that if
$m>0$, then the $\Sp(H)$-module $\Gr^W_{-m}\u_{S,D}$ contains the representation
$[k,k]$ when $m=2k$ and $[k,k,1]$ when $m=2k-1$. If we take $\Gr^M_2\sp(H)$ to
be positive roots of $\sp(H)$, then the highest weight vectors of each of these
representations lies in $\Gr^M_2\Gr^W_{-m}\g_{S,D}$. Since $\gl_g$ is a
subalgebra of $\sp(H)$ with the same Cartan subalgebra, the $\gl_g$-submodule of
$\Gr^M_2\Gr^W_{-m}\g_{S,D}$ generated by $v$ will correspond to the same
partition.
\end{proof}

Using the formula \cite[(6.4)]{fulton-harris}, when $k\ge 2$ we have:
\begin{align*}
\dim V_{[k,k]} &=
\frac{(g-1)(g+k-1)\prod_{j=0}^{k-2}(g+j)^2}{k!(k+1)!}
= \cr
\dim V_{[k,k,1]} &=
\frac{(g-1)(g-2)(g+k-1)\prod_{j=0}^{k-2}(g+j)^2}{(k-1)!(k+2)!}
\end{align*}
where $V_\lambda$ denotes the $\gl_g$-module corresponding to the partition
$\lambda$. These dimensions increase monotonically with $k$. The proof is
completed by an elementary computation, which is left to the reader.

\end{document}